\documentclass{article}
\usepackage{latexsym,amsfonts,amsthm,amsmath,amscd,amssymb}
\usepackage[dvips]{graphicx}

\newtheorem{lemma}{Lemma}[section]
\newtheorem{thm}[lemma]{Theorem}
\newtheorem{rem}[lemma]{Remark}
\newtheorem{prop}[lemma]{Proposition}
\newtheorem{cor}[lemma]{Corollary}

\newtheorem{example}[lemma]{Example}
\newtheorem{defn}[lemma]{Definition}

\newcommand{\cl}{C \kern -0.1em \ell}

\newcommand\matR{{\mathbb{R}}}
\newcommand\matN{{\mathbb{N}}}

\newcommand\calD{{\mathcal D}}

\begin{document}

\title{Diffeological De Rham operators}

\author{Ekaterina~{\textsc Pervova}}

\maketitle

\begin{abstract}
\noindent We consider the notion of the De Rham operator on finite-dimensional diffeological spaces such that the diffeological counterpart $\Lambda^1(X)$ of the cotangent bundle, the so-called 
pseudo-bundle of values of differential 1-forms, has bounded dimension. The operator is defined as the composition of the Levi-Civita connection on the exterior algebra pseudo-bundle 
$\bigwedge(\Lambda^1(X))$ with the standardly defined Clifford action by $\Lambda^1(X)$; the latter is therefore assumed to admit a pseudo-metric for which there exists a Levi-Civita connection. Under 
these assumptions, the definition is fully analogous the standard case, and our main conclusion is that this is the only way to define the De Rham operator on a diffeological space, since we show that there 
is not a straightforward counterpart of the definition of the De Rham operator as the sum $d+d^*$ of the exterior differential with its adjoint. We show along the way that other connected notions do not 
have full counterparts, in terms of the function they are supposed to fulfill, either; this regards, for instance, volume forms, the Hodge star, and the distinction between the $k$-th exterior degree of 
$\Lambda^1(X)$ and the pseudo-bundle of differential $k$-forms $\Lambda^k(X)$.

\end{abstract}

\section*{Introduction}

The concept of a \emph{diffeological space} (introduced in \cite{So1}, \cite{So2}; see \cite{iglesiasBook} for a recent and comprehensive treatment, and also \cite{CSW_Dtopology}, \cite{CWhomotopy}, 
\cite{wu}, \cite{watts}, \cite{iglOrb} for the development of various specific aspects) is a simple and flexible generalization of the concept of a smooth manifold (see \cite{St} for a review of other similar 
directions). Many constructions of differential geometry also generalize, although for some of them there is not (yet) a universal agreement on the choice of the proper counterpart of such-and-such 
notion; this is the case of the tangent bundle, for which there are many proposed versions; the most accepted one at the moment seems to be that of the \emph{internal tangent bundle} \cite{CWtangent} 
(see \cite{HeTangent} for the earlier construction on which it is partially based). Whereas for the cotangent bundle and higher-order differential forms there is a standard version, see for instance 
\cite{iglesiasBook}, \cite{Lau}, \cite{karshon-watts} (as well as \cite{HeCohomology}). Finally, see \cite{magnot1}, \cite{magnot2} for a more analytic context.

A certain development of other concepts of differential geometry on diffeological spaces appears in \cite{vincent}, where (in particular) the basic concept of the diffeological counterpart of a smooth vector bundle 
was developed to some extent. However, the notion itself and its peculiarities with respect to the standard one were already investigated in \cite{iglFibre}; this is where diffeological bundles (which herein we call 
\emph{pseudo-bundles}) were first introduced. The other concepts follow from there, in particular, those needed to define a Dirac operator, which was done in \cite{dirac}. Since diffeological versions of some 
classic instances of Dirac operators were not considered therein, in this paper we try to fill this void, describing the diffeological version of the most classic one of all, the De Rham operator. Our main 
conclusions in this respect are of two sorts. The first is that there does not appear to be any straightforward way of defining a diffeological counterpart of the classical operator $d+d^*$ in the diffeological 
context. This starts from the fact that the differential itself, defined on the spaces $\Omega^k(X)$, does not descend to the pseudo-bundles $\Lambda^k(X)$. Furthermore, the exterior product defined between 
the former vector spaces does not yield an identification between $\bigwedge^k(\Lambda^1(X))$ and $\Lambda^k(X)$, although it gives a natural, possibly surjective, map from the former to the latter. 
We also show that the dimensions of the fibres of $\Lambda^k(X)$ do not truly correlate with the (diffeological) dimension of the space $X$; if $\dim(X)=n$ then $\Lambda^k(X)$ are indeed trivial for $k>n$, 
but for $k\leqslant n$ the dimensions of fibres of $\Lambda^k(X)$ can be arbitrarily large. Finally, since $\Lambda^1(X)$ may have fibres of varying dimension, there does not seem to be a straightforward 
definition of the Hodge star on its exterior degrees.

Diffeological spaces form a very wide category, so that a statement applying to them all would necessarily risk being too general so as to be meaningless. This issue we resolve by dedicating significant 
attention to the diffeological gluing procedure (\cite{pseudobundle}) applied to pairs of diffeological spaces, that in turn satisfy some additional assumptions. Mostly these assumptions have to do with being 
able to put a \emph{pseudo-metric} on the corresponding pseudo-bundles $\Lambda^k(X)$, and with the extendability of differential forms, that we define below. Under these assumptions we do describe 
the behavior of pseudo-bundles $\Lambda^k(X)$ under gluing, as well as that of the De Rham groups.

\paragraph{Acknowledgments} Discussions with Prof. Riccardo Zucchi benefitted significantly this work.

\section{Main definitions}

A recent and comprehensive exposition of the main notions and constructions of diffeology can be found in \cite{iglesiasBook}; that particularly includes the De Rham cohomology (see also 
\cite{HeCohomology}). The homological algebra is discussed in a recent \cite{wu}.

\subsection{Diffeological spaces, pseudo-bundles, and pseudo-metrics}

The concept of a diffeological space is a natural generalization of that of a smooth manifold; briefly, the two differ in that for a diffeological space the notion of atlas is taken by that of a \emph{diffeological 
structure} whose charts have domains of definition of varying dimension. Furthermore, a diffeological space is not subject to the same topological requirements, such as paracompactness etc.

\begin{defn}
A \textbf{diffeological space} is any set $X$ endowed with a \textbf{diffeological structure} (or \textbf{diffeology}), which is the set $\calD$ of maps, called \textbf{plots}, $\calD=\{p:U\to X\}$, for all domains 
$U\subseteq\matR^n$ and for all $n\in\matN$ that satisfies the \emph{covering} condition of every constant map being a plot, the \emph{smooth compatibility} condition of every pre-compostion $p\circ h$ 
of any plot $p:U\to X$ with any ordinary smooth map $h:V\to U$, being again a plot, and the following \emph{sheaf condition}: if $U=\cup_iU_i$ is an open cover of a domain $U$ and $p:U\to X$ is a set map 
such that each restriction $p|_{U_i}:U_i\to X$ is a plot (that is, it belongs to $\calD$) then $p$ itself is a plot.
\end{defn} 

For two diffeological spaces $X$ and $Y$, a set map $f:X\to Y$ is \textbf{smooth} if for every plot $p$ of $X$ the composition $f\circ p$ is a plot of $Y$. If the \emph{vice versa} is always true locally, \emph{i.e.} 
if, whenever the composition of $f\circ p$ with some set map $p:U\to X$ is a plot of $Y$, the map $p$ is necessarily a plot of $X$, and furthermore $f$ is surjective, then $f$ is called a \textbf{subduction}. Said in 
reverse, if, given two diffeological spaces $X$ and $Y$, there exists a subduction $f$ of $X$ onto $Y$, then the diffeology of $Y$ is said to be the \textbf{pushforward} of the diffeology of $X$ by the map $f$. For
instance, if $X$ is a diffeological space and $\sim$ is any equivalence relation on $X$ then the \textbf{quotient diffeology} on $X/\sim$ is defined by the requirement that the quotient projection $X\to X/\sim$ be 
a subduction. Notice in particular that, unlike in the case of smooth manifolds, every quotient of a diffeological space is again a diffeological space. The same is true for any subset $Y\subseteq X$ of a 
diffeological space $X$; it is endowed with the \textbf{subset diffeology} that consists of precisely the plots of $X$ whose ranges are contained in $Y$.

A smooth manifold is an instance of a diffeological space; the corresponding diffeology is given by the set of all usual smooth maps into it. Standard diffeologies are defined for disjoint unions, direct products, 
and spaces of smooth maps between two diffeological spaces (see \cite{iglesiasBook}). For a diffeological space carrying an algebraic structure there is an obvious notion of smoothness of that structure, so 
there are notions of a diffeological vector space, diffeological group, etc. 

The diffeological counterpart of a smooth vector bundle, that we call a \textbf{diffeological vector pseudo-bundle}, is defined analogously to the standard notion, with the exception that there is no requirement 
of there being an atlas of local trivializations. The precise definition is as follows.

\begin{defn}
A \textbf{diffeological vector pseudo-bundle} is a smooth surjective map $\pi:V\to X$ between two diffeological spaces that satisfies the following requirements: 1) for every $x\in X$ the pre-image $\pi^{-1}(x)$ 
carries a vector space structure; 2) the induced operations of addition $V\times_X V\to V$ and scalar multiplication $\matR\times V\to V$ are smooth for the subset diffeology on 
$V\times_X V\subseteq V\times V$, for the product diffeologies on $V\times V$ and $\matR\times V$, and for the standard diffeology on $\matR$; 3) the zero section $X\to V$ is smooth.
\end{defn} 

This notion appeared in \cite{iglFibre}, where it is called \emph{diffeological fibre bundle}, and was considered in \cite{vincent} under the name of \emph{regular vector bundle} and in \cite{CWtangent}, where it 
is termed \emph{diffeological vector space over $X$}. Some developments of the notion appear in \cite{pseudobundle}.

For such pseudo-bundles there are suitable counterparts of all the usual operations on vector bundles, such as direct sums, tensor products, and taking dual bundles. It is worth noting that already in the case 
of (finite-dimensional) diffeological vector spaces the expected notion of duality leads, in general, to different conclusions, specifically the diffeological dual of a vector space may not be isomorphic to the space 
itself. A long-ranging consequence is that there is no proper analogue of a Riemannian metric on a diffeological vector pseudo-bundle, although there is an obvious substitute 
(\cite{pseudobundle-pseudometrics}).

\begin{defn}
Let $\pi:V\to X$ be a diffeological vector pseudo-bundle such that the vector space dimension of each fibre $\pi^{-1}(x)$ is finite. A \textbf{pseudo-metric} on it is a smooth map $g:X\to V^*\otimes V^*$ such that 
for all $x\in X$ the bilinear form $g(x)\in(\pi^{-1}(x))^*\otimes(\pi^{-1}(x))^*$ is symmetric, positive semidefinite, and of rank equal to $\dim((\pi^{-1}(x))^*)$. 
\end{defn}

The reason why this definition is stated as it is, is that in general a finite-dimensional diffeological vector space does not admit a smooth scalar product (\cite{iglesiasBook}). The maximal rank of a smooth 
symmetric bilinear form on such a space is the dimension of its diffeological dual, and there is always a smooth symmetric positive semidefinite form that achieves that rank (\cite{dirac}, Section 5). The latter is 
called a \textbf{pseudo-metric} on the vector space in question, and the notion of a pseudo-metric on a pseudo-bundle is an obvious extension of that. Notice that not every finite-dimensional pseudo-bundle 
admits a pseudo-metric (see \cite{pseudobundle}).

If $\pi:V\to X$ is a finite-dimensional diffeological vector pseudo-bundle and $g$ is a pseudo-metric on it then there is an obvious \textbf{pairing map} 
$$\Phi_g:V\to V^*\,\,\mbox{ given by }\,\,\Phi_g(v)(\cdot)=g(\pi(v))(v,\cdot).$$ This map is a subduction onto $V^*$; it is bijective and a diffeomorphism if and only if the subset diffeology on all fibres of $V$ is the 
standard one, while in general it has a canonically defined right inverse which, however, is not guaranteed to be smooth. The latter is also the reason why the standard construction of the dual $g^*$ via the 
identity
$$g^*(x)(\Phi_g(v),\Phi_g(w))=g(x)(v,w)\,\,\mbox{ for all }\,\,x\in X,\,v,w\in V,$$ although it yields a well-defined family of pseudo-metrics on fibres of $V^*$, may not itself be a pseudo-metric.

\subsection{Diffeological gluing}

The \textbf{diffeological gluing} (\cite{pseudobundle}) is a procedure that mimics the usual topological gluing. Let $X_1$ and $X_2$ be two diffeological spaces, and let $f:X_1\supseteq Y\to X_2$ be a map 
smooth for the subset diffeology on $Y$. The result of the diffeological gluing of $X_1$ to $X_2$ along $f$ is the space
$$X_1\cup_f X_2:=(X_1\sqcup X_2)/_{\sim},\,\,\mbox{ where }x\sim x'\Leftrightarrow x=x'\mbox{ or }f(x)=f(x')$$ endowed with the quotient diffeology of the disjoint union diffeology on $X_1\sqcup X_2$. In 
practice, the plots of $X_1\cup_f X_2$ can be characterized as follows. 

Let us first define the \textbf{standard inductions} $i_1:X_1\setminus Y\to X_1\cup_f X_2$ and $i_2:X_2\to X_1\cup_f X_2$ given as the compositions
$$i_1:X_1\setminus Y\hookrightarrow X_1\sqcup X_2\to X_1\cup_f X_2,\,\,i_2:X_2\hookrightarrow X_1\sqcup X_2\to X_1\cup_f X_2$$ of the obvious inclusions with the quotient projection. Notice that the 
images $i_1(X_1\setminus Y)$ and $i_2(X_2)$ form a disjoint cover of $X_1\cup_f X_2$, a property that is used to describe maps from/into $X_1\cup_f X_2$. For instance, the plots of $X_1\cup_f X_2$ can 
be given the following characterization. A map $p:U\to X_1\cup_f X_2$ defined on a connected domain $U$ is a plot of $X_1\cup_f X_2$ if and only if one of the following is true: either there exists a plot 
$p_1$ of $X_1$ such that 
$$p(u)=\left\{\begin{array}{ll} p_1(u) & \mbox{if }u\in p_1^{-1}(X_1\setminus Y), \\ i_2(f(p_1(u))) & \mbox{if }u\in p_1^{-1}(Y), \end{array}\right.$$ or there exists a plot $p_2$ of $X_2$ such that 
$$p=i_2\circ p_2.$$

The right-hand factor $X_2$ always embeds into $X_1\cup_f X_2$, while $X_1$ in general does not, unless $f$ is a diffeomorphism (which is the case we will mostly treat). If it is one then the map
$$\tilde{i}_1:X_1\to X_1\cup_f X_2\,\,\mbox{ given by }\,\,\tilde{i}_1(x_1)=\left\{\begin{array}{ll} i_1(x_1) & \mbox{if }x_1\in X_1\setminus Y, \\ i_2(f(x_1)) & \mbox{if }x_1\in Y \end{array}\right.$$ is also an inclusion. 

Suppose now that we are given two pseudo-bundles $\pi_1:V_1\to X_1$ and $\pi_2:V_2\to X_2$, a gluing map $f:X_1\supseteq Y\to X_2$, and a smooth lift $\tilde{f}:\pi_1^{-1}(Y)\to V_2$ of $f$, that is linear 
on each fibre in its domain of definition. Then $\tilde{f}$ defines a gluing of $V_1$ to $V_2$ that preserves the pseudo-bundle structures, and specifically, we obtain in an obvious way a new pseudo-bundle 
denoted by
$$\pi_1\cup_{(\tilde{f},f)}\pi_2:V_1\cup_{\tilde{f}}V_2\to X_1\cup_f X_2.$$ The standard inductions $V_1\setminus\pi_1^{-1}(Y)\to V_1\cup_{\tilde{f}}V_2$ and $V_2\to V_1\cup_{\tilde{f}}V_2$ are denoted by 
$j_1$ and $j_2$ respectively. 

The gluing of pseudo-bundles is well-behaved with respect to the operations of direct sum and tensor product, while for dual pseudo-bundles its behavior is more complicated, unless both $\tilde{f}$ and $f$ 
are diffeomorphisms (see \cite{pseudobundle} and \cite{pseudobundle-pseudometrics}). Certain pairs of pseudo-metrics on $V_1$ and $V_2$ allow to obtain a pseudo-metric on $V_1\cup_{\tilde{f}}V_2$.

\begin{defn}
Let $\pi_1:V_1\to X_1$ and $\pi_2:V_2\to X_2$ be two finite-dimensional diffeological vector pseudo-bundles, let $(\tilde{f},f)$ be a gluing between them, and let $g_1$ and $g_2$ be pseudo-metrics on $V_1$ 
and $V_2$ respectively. The pseudo-metrics $g_1$ and $g_2$ are said to be \textbf{compatible} (with the gluing along $(\tilde{f},f)$) if for all $y\in Y$ and for all $v_1,w_1\in\pi_1^{-1}(y)$ we have 
$$g_1(y)(v_1,w_1)=g_2(f(y))(\tilde{f}(v_1),\tilde{f}(w_1)).$$ If $g_1$ and $g_2$ are compatible then the \textbf{induced pseudo-metric $\tilde{g}$ on $V_1\cup_{\tilde{f}}V_2$} is defined by
$$\tilde{g}(x)(v,w)=\left\{\begin{array}{ll} g_1(i_1^{-1}(x))(j_1^{-1}(v),j_1^{-1}(w)) & \mbox{if }x\in i_1(X_1\setminus Y), \\ 
g_2(i_2^{-1}(x))(j_2^{-1}(v),j_2^{-1}(w)) & \mbox{if }x\in i_2(X_2)\end{array}\right.$$ for all $x\in X_1\cup_f X_2$ and for all $v,w\in(\pi_1\cup_{(\tilde{f},f)}\pi_2)^{-1}(x)$. 
\end{defn}

See \cite{pseudobundle-pseudometrics} for details.

\subsection{Differential forms, diffeological connections, and Levi-Civita connections}

The notion of a diffeological differential form is a rather well-developed one by now, see \cite{iglesiasBook}; it is defined as a collection of usual differential forms satisfying a certain compatibility condition. 
Namely, let $X$ be a diffeological space, and let $\calD$ be its diffeology. A \textbf{diffeological differential $k$-form on $X$} is a collection $\omega=\{\omega(p)\}_{p\in\calD}$, where $p:U\to X$ with 
$U\subseteq\matR^n$ a domain and $\omega(p)\in C^{\infty}(U,\Lambda^k(\matR^n))$, such that for any ordinary smooth map $F:V\to U$ defined on another domain $V$ and with values in $U$ we have 
that $\omega(p\circ F)=F^*(\omega(p))$. The collection of all such forms for a fixed $k$, denoted by $\Omega^k(X)$, is a real vector space and is endowed with the diffeology given by the following condition:
a map $q:V\to\Omega^k(X)$ is a plot of $\Omega^k(X)$ is a plot of $\Omega^k(X)$ if and only if for every plot $p:\matR^n\subseteq U\to X$ the map 
$$V\times U\ni(v,u)\mapsto q(v)(p(u))\in\Lambda^k(\matR^n)$$ is smooth in the usual sense. 

A specific example of a diffeological differential form on $X$ is the \textbf{differential} of a smooth function $h:X\to\matR$, where $\matR$ is considered with the standard diffeology. The differential $dh$ is 
defined by setting
$$dh(p)=d(h\circ p)\,\,\mbox{ for every plot }p:U\to X,$$ where $d(h\circ p)$ is the usual differential of an ordinary smooth function $U\to\matR$. Checking that $dh$ is well-defined as an element of 
$\Omega^1(X)$ is trivial.

The definition of $\Omega^k(X)$ then extends to that of the \textbf{pseudo-bundle $\Lambda^k(X)$ of $k$-forms over $X$} (termed the \emph{bundle of values of $k$-forms} on $X$ in \cite{iglesiasBook}), 
in the following way. We first define, for every $x\in X$, the space $\Omega_x^k(X)$ of \textbf{$k$-forms on $X$ vanishing at $x$}. A form $\omega\in\Omega^k(X)$ \textbf{vanishes at $x$} if for every plot 
$p:U\to X$ of $X$ such that $U\ni 0$ and $p(0)=x$ we have that $\omega(p)(0)=0$, the zero form; the set of all such $k$-forms is the subspace $\Omega_x^k(X)$, which is indeed a vector subspace of 
$\Omega^k(X)$ and is endowed with the subset diffeology. Consider next the trivial pseudo-bundle $X\times\Omega^k(X)$ over $X$. The union $\bigcup_{x\in X}\{x\}\times\Omega_x^k(X)$ is a sub-bundle 
of $X\times\Omega^k(X)$ in the sense of diffeological vector pseudo-bundles, so the corresponding quotient pseudo-bundle is again a diffeological vector pseudo-bundle; $\Lambda^k(X)$ is precisely this 
pseudo-bundle:
$$\Lambda^k(X):=(X\times\Omega^k(X))/\left(\bigcup_{x\in X}\{x\}\times\Omega_x^k(X)\right).$$ The quotient projection is denoted by $\pi^{\Omega^k,\Lambda^k}$. 

In particular, if $k=1$ the pseudo-bundle $\Lambda^1(X)$ acts as a substitute of the usual cotangent bundle. Indeed, if $X$ is a smooth manifold considered as a diffeological space for the standard diffeology 
of a smooth manifold (see above), $\Lambda^1(X)$ coincides naturally with the cotangent bundle $T^1(X)$. Thus, a \textbf{diffeological connection on a pseudo-bundle $\pi:V\to X$} is defined as an operator 
$$\nabla:C^{\infty}(X,V)\to C^{\infty}(X,\Lambda^1(X)\otimes V),$$ satisfying then the usual properties of linearity and the Leibnitz rule.

\begin{defn}
Let $\pi:V\to X$ be a diffeological vector pseudo-bundle. A \textbf{diffeological connection on $V$} is a smooth linear operator
$$\nabla:C^{\infty}(X,V)\to C^{\infty}(X,\Lambda^1(X)\otimes V)$$ such that for all $h\in C^{X,\matR}$ and for all $s\in C^{\infty}(X,V)$ we have
$$\nabla(hs)=dh\otimes s+h\nabla s,$$ where $dh\in\Lambda^1(X)$ is defined by $dh(x)=\pi^{\Omega,\Lambda}(x,dh)$, where $dh$ on the right-hand side is the already-defined differential 
$dh\in\Omega^1(X)$.
\end{defn}

A particular instance of a diffeological connection is the \textbf{Levi-Civita connection} on $\Lambda^1(X)$ endowed with a pseudo-metric $g^{\Lambda}$. Two assumptions are implicit in this notion: that 
$X$ is such that $\Lambda^1(X)$ has finite-dimensional fibres, and that $\Lambda^1(X)$ admits a pseudo-metric. If it does then the following definition (\cite{connectionsLC}) is well-posed.

\begin{defn}
Let $X$ be a diffeological space such that $\Lambda^1(X)$ admits pseudo-metrics, and let $g^{\Lambda}$ be a pseudo-metric on $\Lambda^1(X)$. A \textbf{Levi-Civita connection} on 
$(\Lambda^1(X),g^{\Lambda})$ is a connection $\nabla$ on $\Lambda^1(X)$ which satisfies the usual two conditions. Specifically, $\nabla$ is \textbf{compatible with the pseudo-metric $g^{\Lambda}$}, that is, 
for any two sections $s,t\in C^{\infty}(X,\Lambda^1(X))$
$$d(g^{\Lambda}(s,t))=g^{\Lambda}(\nabla s,t)+g^{\Lambda}(s,\nabla t),$$ where on the left we have the differential of $g^{\Lambda}(s,t)\in C^{\infty}(X,\matR)$ that is an element of $C^{\infty}(X,\Lambda^1(X))$ 
and $g^{\Lambda}$ is extended to sections of $\Lambda^1(X)\otimes\Lambda^1(X)$ by setting $g^{\Lambda}(\alpha\otimes s,t):=\alpha\cdot g^{\Lambda}(s,t)=g^{\Lambda}(s,\alpha\otimes t)$ for any 
$\alpha\in C^{\infty}(X,\Lambda^1(X))$. Second, $\nabla$ is \textbf{symmetric}, that is, for any $s,t\in C^{\infty}(X,\Lambda^1(X))$ we have 
$$\nabla_s t-\nabla_t s=[s,t],$$ where $\nabla_s t$ is the \textbf{covariant derivative of $t$ along $s$} and $[s,t]$ is the \textbf{Lie bracket} of $s$ and $t$, both of which are defined via the pairing map 
$\Phi_{g^{\Lambda}}$ corresponding to the pseudo-metric $g^{\Lambda}$. 
\end{defn}

The (very few, this is a straightforward extension of the standard notion) details concerning the definitions of covariant derivatives and the Lie bracket can be found in \cite{dirac}, Sections 10.2 and 11.1. 
It is not quite clear when $X$ admits a Levi-Civita connection, but if it does, it is unique.

The pseudo-bundles of differential forms are rather well-behaved with respect to the gluing, provided that certain \emph{extendibility conditions} are satisfied (see \cite{dirac}, Section 8.1, for the case of $k=1$), 
and as a consequence, the same is true for diffeological connections and the Levi-Civita connections. Specifically, given two connections $\nabla^1$ and $\nabla^2$ on pseudo-bundles $\pi_1:V_1\to X_1$ 
and $\pi_2:V_2\to X_2$ yield a well-defined connection on $\pi_1\cup_{(\tilde{f},f)}\pi_2:V_1\cup_{\tilde{f}}V_2\to X_1\cup_f X_2$, as long as $\nabla^1$ and $\nabla^2$ satisfy a certain compatibility 
condition with respect to the gluing along $(\tilde{f},f)$, and $\Lambda^1(X_1)$ and $\Lambda^1(X_2)$ satisfy (one of) the already-mentioned extendibility conditions relative to $f$. Furthermore, if 
$\Lambda^1(X_1)$ and $\Lambda^1(X_2)$ satisfy the extendibility condition and are endowed each with a connection then under a certain additional condition (this is also called a compatibility condition, 
but it is a different one from that in the case of $V_1\cup_{\tilde{f}}V_2$, see \cite{dirac}, Section 11.4.1; compare with \cite{dirac}, Section 10.3.1) two connections on $\Lambda^1(X_1)$ and $\Lambda^1(X_2)$ 
yield a well-defined connection $\nabla^{\Lambda}$ on $\Lambda^1(X_1\cup_f X_2)$. Moreover, if $\Lambda^1(X_1)$ and $\Lambda^1(X_2)$ are endowed with pseudo-metrics $g_1^{\Lambda}$ and 
$g_2^{\Lambda}$ well-behaved (see \cite{dirac}, Section 8.4.2, for definition) with respect to $f$, and the initial connections on them are the Levi-Civita connections then $\nabla^{\Lambda}$ is the 
Levi-Civita connection on $\Lambda^1(X_1\cup_f X_2)$ endowed with a certain induced pseudo-metric $g^{\Lambda}$ (\cite{dirac}, Section 8.4.3).

\subsection{Pseudo-bundles of Clifford modules, diffeological Clifford connections, and Dirac operators}

As we have mentioned already, the operations of direct sums, tensor products, and quotienting are defined also for diffeological vector pseudo-bundles; this in particular allows to obtain a well-defined 
pseudo-bundle $\pi^{\cl}:\cl(V,g)\to X$ of Clifford algebras starting from a given pseudo-bundle $\pi:V\to X$ endowed with a pseudo-metric $g$. Each fibre of $\cl(V,g)$ is the Clifford algebra $\cl(\pi^{-1}(x),g(x))$. 
It then makes sense to speak of another pseudo-bundle $\chi:E\to X$ over the same $X$ being a \textbf{pseudo-bundle of Clifford modules} over $\cl(V,g)$, in the sense that each fibre $\chi^{-1}(x)$ is a 
Clifford module over $\cl(\pi^{-1}(x),g(x))$ with some Clifford action $c(x)$. For $E$ to be a pseudo-bundle of Clifford modules, it suffices to add the requirement that the total action $c$ be smooth. This condition 
of smoothness can be stated as follows: for every plot $q:U'\to\cl(V,g)$ and for every plot $p:U\to E$ the map
$$U'\times U\supseteq\{(u',u)\,|\,\pi^{cl}(q(u'))=\chi(p(u))\}\mapsto c(q(u'))(p(u))\in E$$ is smooth for the subset diffeology on its domain of definition. 

Given then a diffeological space $X$ such that $\Lambda^1(X)$ admits a pseudo-metric $g^{\Lambda}$ such that there exists the Levi-Civita connection $\nabla^{\Lambda}$ on $(\Lambda^1(X),g^{\Lambda})$, 
and given a pseudo-bundle of Clifford modules $\chi:E\to X$ over $\cl(\Lambda^1(X),g^{\Lambda})$ with Clifford action $c$, the notion of a Clifford connection on $E$ is well-defined (although its existence 
is not guaranteed).

\begin{defn}
A connection $\nabla$ on $E$ is a \textbf{Clifford connection} if for every $s,t\in C^{\infty}(X,\Lambda^1(X))$ and for every $r\in C^{\infty}(X,E)$ we have 
$$\nabla_t(c(s)r)=c(\nabla^{\Lambda}_t s)(r)+c(s)(\nabla_t r).$$
\end{defn}

This is quite the same as the standard notion, just using the diffeological counterparts of all components. Then the composition $c\circ\nabla$ of a given Clifford action with the given Clifford connection is, 
as usual, a \textbf{Dirac operator} on $E$.

\begin{defn}
Let $X$ be a diffeological space such that $\Lambda^1(X)$ admits a pseudo-metric $g^{\Lambda}$ and there exists a Levi-Civita connection on $(\Lambda^1(X),g^{\Lambda})$. Let $\chi:E\to X$ be a 
pseudo-bundle of Clifford modules over $\cl(\Lambda^1(X),g^{\Lambda})$ with Clifford action $c$, and let $\nabla$ be a Clifford connection on $E$. Associated to the data $(X,g^{\Lambda},E,c,\nabla)$ is the 
\textbf{Dirac operator} $D:C^{\infty}(X,E)\to C^{\infty}(X,E)$ given by $D=c\circ\nabla$.
\end{defn}

All these constructions are well-behaved with respect to gluing, provided that all gluing maps are diffeological diffeomorphisms, and that certain compatibility and extendibility conditions are met. Specifically, 
given two pseudo-bundles $\chi_1:E_1\to X_1$ and $\chi_2:E_2\to X_2$ of Clifford modules over $(\Lambda^1(X_1),g_1^{\Lambda})$ and $(\Lambda^1(X_2),g_2^{\Lambda})$ with Clifford actions $c_1$ 
and $c_2$, that are endowed with Clifford connections $\nabla^1$ (on $E_1$) and $\nabla^2$ (on $E_2$), and given a gluing of $E_1$ to $E_2$, along a pair of diffeomorphisms $f:X_1\supseteq Y\to X_2$ 
and $\tilde{f}':E_1\supseteq\chi_1^{-1}(Y)\to E_2$, we need the following conditions for there being a well-defined Dirac operator on the result of gluing:
\begin{enumerate}
\item The map $f$ is such that the following two diffeologies on $\Omega^1(Y)$ coincide: the pushforward $\calD_1^{\Omega}$ of the standard diffeology on $\Omega^1(X_1)$ by the pullback map 
$i^*:\Omega^1(X_1)\to\Omega^1(Y)$, where $i:Y\hookrightarrow X_1$ is the natural inclusion, and the pushforward $\calD_2^{\Omega}$ of the standard diffeology on $\Omega^1(X_2)$ by the pullback map 
$j^*:\Omega^1(X_2)\to\Omega^1(f(Y))$, where $j:f(Y)\hookrightarrow X_2$ is also the natural inclusion. The equality $\calD_1^{\Omega}=\calD_2^{\Omega}$ is what we previously called the extendibility 
condition, and it ensures that $\Lambda^1(X_1\cup_f X_2)$ admits a particularly simple description in terms of $\Lambda^1(X_1)$ and $\Lambda^1(X_2)$ (it is possible to give a description without the 
extendibility condition, but it is far more cumbersome). See \cite{dirac}, Section 8, for details;
\item The pseudo-metrics $g_1^{\Lambda}$ and $g_2^{\Lambda}$ are \textbf{compatible} with the gluing along $f$, that is, for every $y\in Y$ and for every pair $\alpha_1\in\Lambda_y^1(X_1)$, 
$\alpha_2\in\Lambda_{f(y)}^1(X_2)$ such that $i_{\Lambda}^*(\alpha_1)=(f_{\Lambda}^*j_{\Lambda}^*)(\alpha_2)$ (we say that $\alpha_1$ and $\alpha_2$ are \textbf{compatible}), where 
$i_{\Lambda}^*:\Lambda^1(X_1)\supseteq(\pi_1^{\Lambda})^{-1}(Y)\to\Lambda^1(Y)$, $f_{\Lambda}^*:\Lambda^1(f(Y))\to\Lambda^1(Y)$, 
$j_{\Lambda}^*:\Lambda^1(X_2)\supseteq(\pi_2^{\Lambda})^{-1}(f(Y))\to\Lambda^1(f(Y))$ are induced by the pullback maps $i^*$, $f^*$, and $j^*$, we have that 
$$g_1^{\Lambda}(y)(\alpha_1,\alpha_1)=g_2^{\Lambda}(f(y))(\alpha_2,\alpha_2);$$
\item The actions $c_1$ and $c_2$ are compatible with $(\tilde{f}',f)$, specifically, for every $y\in Y$, for every compatible pair $\alpha_1\in\Lambda_y^1(X_1)$, $\alpha_2\in\Lambda_{f(y)}^1(X_2)$, and for 
every $e_1\in\chi_1^{-1}(y)$ we have that
$$c_2(\alpha_2)(\tilde{f}'(e_1))=\tilde{f}'(c_1(\alpha_1)(e_1));$$
\item The pseudo-bundles $(\Lambda^1(X_1),g_1^{\Lambda})$ and $(\Lambda^1(X_2),g_2^{\Lambda})$ admit Levi-Civita connections $\nabla^{\Lambda,1}$ and $\nabla^{\Lambda,2}$, and these 
connections are \textbf{compatible} in the following sense: for all $t_1\in C^{\infty}(X_1,\Lambda^1(X_1))$ and $t_2\in C^{\infty}(X_2,\Lambda^1(X_2))$ such that for all $\in Y$ we have that 
$i_{\Lambda}^*(s_1(y))=(f_{\Lambda}^*j_{\Lambda}^*)(s_2(f(y)))$ (these are \textbf{compatible} sections of $\Lambda^1(X_1)$ and $\Lambda^1(X_2)$), the following equality holds at every point $yin Y$:
$$(i_{\Lambda}^*\otimes i_{\Lambda}^*)((\nabla^{\Lambda,1}s_1)(y))=((f_{\Lambda}^*j_{\Lambda}^*)\otimes(f_{\Lambda}^*j_{\Lambda}^*))((\nabla^{\Lambda,2}s_2)(f(y)));$$
\item The connections $\nabla^1$ and $\nabla^2$ are \textbf{compatible} with the gluing along $(\tilde{f}',f)$, which means the following: for every pair $r_1\in C^{\infty}(X_1,E_1)$, $r_2\in C^{\infty}(X_2,E_2)$ 
such that for every $y\in Y$ we have $\tilde{f}'(s_1(y))=s_2(f(y))$, and for all $y\in Y$ there is the equality
$$(i_{\Lambda}^*\otimes\tilde{f}')((\nabla^1s_1)(y))=((f_{\Lambda}^*j_{\Lambda}^*)\otimes\mbox{Id}_{E_2})((\nabla^2s_2)(f(y))).$$
\end{enumerate}

The conditions just listed provide us with the following:
\begin{enumerate}
\item Conditions 1 and 2 yield an \textbf{induced pseudo-metric} $g^{\Lambda}$ on $\Lambda^1(X_1\cup_f X_2)$;
\item Condition 3 yields the Levi-Civita connection $\nabla^{\Lambda}$ on $(\Lambda^1(X_1\cup_f X_2),g^{\Lambda})$;
\item Condition 4 provides an \textbf{induced Clifford action} $\tilde{c}$ of $\cl(\Lambda^1(X_1\cup_f X_2),g^{\Lambda})$ on $E_1\cup_{\tilde{f}'}E_2$;
\item Condition 5 ensures that there is the induced connection $\nabla^{\cup}$ on $E_1\cup_{\tilde{f}'}E_2$, and that it is a Clifford connection.
\end{enumerate}

\begin{prop} \emph{(\cite{dirac}, Proposition 13.3)}
Let $D_1$ and $D_2$ be the Dirac operators associated to the data $(X_1,g_1^{\Lambda},E_1,c_1,\nabla^1)$ and $(X_2,g_2^{\Lambda},E_2,c_2,\nabla^2)$ respectively, and suppose that these data and 
the gluing pair $(\tilde{f}',f)$ satisfy Conditions 1-5 above. The the Dirac operator $D$ associated to the data $(X_1\cup_f X_2,g^{\Lambda},E_1\cup_{\tilde{f}'}E_2,\tilde{c},\nabla^{\cup})$ satisfies the following: 
for every $s\in C^{\infty}(X_1\cup_f X_2,E_1\cup_{\tilde{f}'}E_2)$ we have that 
$$D(s)=D_1(s_1)\cup_{(f,\tilde{f}')}D_2(s_2),$$ where $s_1:=\tilde{j}_1^{-1}\circ s\circ\tilde{i}_1\in C^{\infty}(X_1,E_1)$ and $s_2:=j_2^{-1}\circ s\circ i_2\in C^{\infty}(X_2,E_2)$. 
\end{prop}

The map $\tilde{j}_1$ above is the natural inclusion $E_1\hookrightarrow E_1\cup_{\tilde{f}'}E_2$, the analogue of the inclusion $\tilde{i}_1$ (recall that also $\tilde{f}'$ is assumed to be a diffeomorphism), and 
the sign $\cup_{(f,\tilde{f}')}$ refers to the gluing of the maps $D_1(s_1):X_1\to E_1$ and $D_2(s_2):X_2\to E_2$ along $(f,\tilde{f}')$; see \cite{dirac}, Section 6.3, for details.

\subsection{Diffeological De Rham cohomology}

There is an established notion of the De Rham cohomology for diffeological spaces; a complete exposition can be found in \cite{iglesiasBook}, Section 6.73. The construction mimics the standard one and 
is as follows. Let $X$ be a diffeological space. The already-defined differential of a smooth function $X\to\matR$ provides us with the \textbf{coboundary operator}
$$d:\Omega^k(X)\to\Omega^{k+1}(X)\,\,\mbox{ for }\,k\geqslant 0,$$ defined by $d\omega(p)=d(\omega(p))$ for any plot $p$ of $X$. This is well-defined and satisfies the coboundary condition $d\circ d=0$, 
see \cite{iglesiasBook}. Define, as usual, the space of \textbf{$k$-cocycles} to be 
$$Z_{dR}^k(X):=\mbox{ker}(d:\Omega^k(X)\to\Omega^{k+1}(X)),$$ and let
$$B_{dR}^k(X):= d(\Omega^{k-1}(X))\subseteq Z_{dR}^k(X)\,\,\mbox{ for }k\geqslant 1,\,\,\mbox{ and }\,B_{dR}^0(X)=\{0\}$$ be the space of \textbf{$k$-coboundaries}. In particular, every $Z_{dR}^k(X)$ is 
equipped with the subset diffeology relative to the standard diffeology on the corresponding $\Omega^k(X)$.

The \textbf{de Rham cohomology groups} are then defined as quotients
$$H_{dR}^k(X):=Z_{dR}^k(X)/B_{dR}^k(X).$$ They are equipped with the quotient diffeology, with respect to which they become diffeological vector spaces.

\section{The pseudo-bundles $\Lambda^k(X_1\cup_f X_2)$, and the groups $H_{dR}^k(X_1\cup_f X_2)$}

In this section we consider the behavior of $\Lambda^k(X)$ and $H_{dR}^k(X)$ under gluing. The common prerequisite for considering this is to describe first the behavior of the spaces $\Omega^k(X)$ 
with respect to gluing (as has already been done for $k=1$, \cite{dirac}, Section 8).

\subsection{The vector spaces $\Omega^k(X_1\cup_f X_2)$} 

As in the case of $k=1$, the spaces $\Omega^k(X_1\cup_f X_2)$ are subspaces of the direct sum $\Omega^k(X_1)\oplus\Omega^k(X_2)$. They can be described as the images of the pullback map 
$$\pi^*:\Omega^k(X_1\cup_f X_2)\to\Omega^k(X_1\sqcup X_2)\cong\Omega^k(X_1)\oplus\Omega^k(X_2),$$ where $\pi:X_1\sqcup X_2\to X_1\cup_f X_2$ is the quotient projection that defines 
$X_1\cup_f X_2$, and also given an explicit description in terms of an appropriate compatibility notion. Doing so does not require any additional assumptions on $f$, which appear when we want to establish 
the surjectivity of the images of the direct sum projections $\pi_1^{\Omega}:\Omega^k(X_1\cup_f X_2)\to\Omega^k(X_1)$ and $\pi_2^{\Omega}:\Omega^k(X_1\cup_f X_2)\to\Omega^k(X_2)$.

\subsubsection{The diffeomorphism $\Omega^k(X_1\sqcup X_2)\cong\Omega^k(X_1)\oplus\Omega^k(X_2)$}

The existence (and the construction) of this diffeomorphism is essentially obvious from the definitions. Let $\hat{i}_1:X_1\hookrightarrow X_1\sqcup X_2$ and $\hat{i}_2:X_2\hookrightarrow X_1\sqcup X_2$ 
be the obvious inclusions.

\begin{thm}\label{omega:of:disjoint:union:splits:as:direct:sum:thm}
For any two diffeological spaces $X_1$ and $X_2$ and for any $k\geqslant 0$ the map 
$$\hat{i}_1^*\oplus\hat{i}_2^*:\Omega^k(X_1\sqcup X_2)\to\Omega^k(X_1)\oplus\Omega^k(X_2)$$ acting by $(\hat{i}_1^*\oplus\hat{i}_2^*)(\omega)=\hat{i}_1^*(\omega)\oplus\hat{i}_2^*(\omega)$ is a 
linear diffeomorphism.
\end{thm}

\begin{proof}
It suffices to show that $\hat{i}_1^*\oplus\hat{i}_2^*$ has a smooth linear inverse. This inverse is given by assigning to each pair $\omega_1\oplus\omega_2$, where $\omega_1\in\Omega^k(X_1)$ and 
$\omega_2\in\Omega^k(X_2)$, the form $\omega$ that is defined as follows. Let $p:U\to X_1\sqcup X_2$ be a plot; then there exists a decomposition $U=U_1\sqcup U_2$ of the domain $U$ as a disjoint 
union of two domains $U_1$ and $U_2$ such that $p_1:=\hat{i}_1^{-1}\circ p|_{U_1}$ and $p_2:=\hat{i}_2^{-1}\circ p|_{U_2}$. We define 
$$\omega(p)=(\omega_1(p_1),\omega_2(p_2)),$$ the latter pair being naturally seen as a usual differential $k$-form on the disjoint union $U_1\sqcup U_2=U$. That such assignment defines the inverse 
of $\hat{i}_1^*\oplus\hat{i}_2^*$, and that this inverse is smooth and linear, is immediate from the construction.
\end{proof}

\subsubsection{The subspace $\Omega_f^k(X_1)$ of $f$-invariant $k$-forms}

Let $f:X_1\supseteq Y\to X_2$. In general, the $k$-forms on $X_1$ which can be carried forward to the glued space $X_1\cup_f X_2$ must satisfy a certain additional condition.

\begin{defn}
Two plots $p_1:U\to X_1$ and $p_1':U'\to X_1$ are called \textbf{$f$-equivalent} if $U=U'$, and for every $u\in U$ such that $p_1(u)\neq p_1'(u)$ we have that $p_1(u),p_1'(u)\in Y$ and $f(p_1(u))=f(p_1'(u))$. 
A $k$-form $\omega_1\in\Omega^k(X_1)$ is said to be \textbf{$f$-invariant} if for any two $f$-equivalent plots $p_1$ and $p_1'$ of $X_1$ we have that 
$$\omega_1(p_1)=\omega_1(p_1').$$ The set of all $f$-invariant $k$-forms on $X_1$ is denoted by $\Omega_f^k(X_1)$.
\end{defn}

It is trivial to establish the following statement (whose proof we therefore omit).

\begin{lemma}
For every diffeological space $X_1$ and for every smooth map $f$ defined on a subset of $X_1$ the set $\Omega_f^k(X_1)$ is a vector subspace of $\Omega^k(X_1)$. 
\end{lemma}

\subsubsection{The inverse of the pullback map $\pi^*$}

Using the diffeomorphism of Theorem \ref{omega:of:disjoint:union:splits:as:direct:sum:thm}, we can now describe the inverse of the ($k$th) pullback map $\pi^*$ as a map on the subspace of 
$\Omega^k(X_1)\oplus\Omega^k(X_2)$ determined by the following condition.

\begin{defn}\label{compatible:forms:defn}
Let $X_1$ and $X_2$ be two diffeological spaces, let $f:X_1\supseteq Y\to X_2$ be a smooth map, and let $k\geqslant 0$. Two forms $\omega_1\in\Omega^k(X_1)$ and $\omega_2\in\Omega^k(X_2)$ are 
said to be \textbf{compatible} if for every plot $p_1$ of the subset diffeology on $Y$ we have 
$$\omega_1(p_1)=\omega_2(f\circ p_1).$$ We denote by 
$$\Omega^k(X_1)\oplus_{comp}\Omega^k(X_2)=\{\omega_1\oplus\omega_2\,|\,\omega_1\mbox{ and }\omega_2\mbox{ are compatible}\}$$ the subset in $\Omega^k(X_1)\oplus\Omega^k(X_2)$ that 
consists of all pairs of compatible forms.
\end{defn}

We define next the map 
$$\mathcal{L}^k:\Omega_f^k(X_1)\oplus_{comp}\Omega^k(X_2)\to\Omega^k(X_1\cup_f X_2)$$ given by setting, for every plot $p:U\to X_1\cup_f X_2$ defined on a connected $U$, 
$$\mathcal{L}^k(\omega_1\oplus\omega_2)(p)=\left\{\begin{array}{ll} 
\omega_1(p_1) & \mbox{if }p=\hat{i}_1\circ p_1\mbox{ for some plot }p_1\mbox{ of }X_1,\\
\omega_2(p_2) & \mbox{if }p=i_2\circ p_2\mbox{ for some plot }p_2\mbox{ of }X_2.
\end{array}\right.$$

\begin{lemma}
For any two diffeological spaces $X_1$ and $X_2$ and for every smooth map $f:X_1\supseteq Y\to X_2$ the map $\mathcal{L}^k$ is well-defined.
\end{lemma}

\begin{proof}
We need to show that $\mathcal{L}^k(\omega_1\oplus\omega_2)(p)$ does not depend on the choice of the lift of $p$ to a plot $p_i$ of $X_i$, and that the assignment 
$p\mapsto\mathcal{L}^k(\omega_1\oplus\omega_2)(p)$ satisfies the smooth compatibility condition. The former of these claims is obvious if $p$ lifts to a plot of $X_2$; indeed, since $i_2$ is injective, such a 
lift is unique. Let $p_1$ and $p_1'$ be two lifts of $p$ to some plots of $X_1$. Then they are obviously $f$-equivalent. Since $\omega_1$ is $f$-invariant by assumption, we have that 
$\omega_1(p_1)=\omega_1(p_1')$, which implies that $\mathcal{L}^k(\omega_1\oplus\omega_2)(p)$ is well-defined.

Let us now show that $\mathcal{L}^k(\omega_1\oplus\omega_2)$ satisfies a smooth compatibility condition. Let $h:U'\to U$ be an ordinary smooth map; then either 
$\mathcal{L}^k(\omega_1\oplus\omega_2)(p\circ h)=\omega_1(p_1\circ h)=h^*(\omega_1(p_1))=h^*(\mathcal{L}^k(\omega_1\oplus\omega_2)(p))$ or 
$\mathcal{L}^k(\omega_1\oplus\omega_2)(p\circ h)=\omega_2(p_2\circ h)=h^*(\omega_2(p_2))=h^*(\mathcal{L}^k(\omega_1\oplus\omega_2)(p))$, and we deduce the smooth compatibility condition for 
$\mathcal{L}^k(\omega_1\oplus\omega_2)$ from those for $\omega_1$ and $\omega_2$ respectively.
\end{proof}

The map $\mathcal{L}^k$ is therefore well-defined, and it is quite obvious that it is linear.

\begin{thm}\label{cal:L:is:inverse:of:pullback:pi:thm}
The map $\mathcal{L}^k$ is a smooth inverse of the pullback map $\pi^*:\Omega^k(X_1\cup_f X_2)\to\Omega^k(X_1\sqcup X_2)\cong\Omega^k(X_1)\oplus\Omega^k(X_2)$.
\end{thm}

\begin{proof}
Let $\omega_1\oplus\omega_2\in\mbox{Range}(\pi^*)$, and let us show that $\omega_1$ and $\omega_2$ are compatible, and that $\omega_1$ is $f$-invariant. Let $p_1$ be a plot for the subset diffeology 
on $Y$; it is thus a plot of $X_1$, and $f\circ p_1$ is a plot of $X_2$. To both of them there corresponds a plot $p$ of $X_1\cup_f X_2$ given by $p=i_2\circ f\circ p_1=\hat{i}_1\circ p_1$. Since 
$\omega_1\oplus\omega_2$ is in the range of $\pi^*$, it is the image $\pi^*\omega$ of some $\omega\in\Omega^k(X_1\cup_f X_2)$. The forms $\omega_1$ and $\omega_2$ are given by 
$$\omega_1(p_1)=\omega(\hat{i}_1\circ p_1)\,\,\mbox{ and }\,\,\omega_2(p_2)=\omega(i_2\circ p_2)$$ respectively (for any arbitrary plots $p_1$ of $X_1$ and $p_2$ of $X_2$. Thus, in the present case we 
have 
$$\omega_1(p_1)=\omega(\hat{i}_1\circ p_1)=\omega(p)=\omega(i_2\circ f\circ p_1)=\omega_2(f\circ p_1),$$ which implies the compatibility of $\omega_1$ and $\omega_2$. 

Suppose now that $p_1$ and $p_1'$ are two $f$-equivalent plots. Then obviously $\hat{i}_1\circ p_1=\hat{i}_1\circ p_1'$, therefore we have 
$$\omega_1(p_1)=\omega(\hat{i}_1\circ p_1)=\omega(\hat{i}_1\circ p_1')=\omega_1(p_1'),$$ that is, $\omega_1$ is $f$-invariant. In particular, we conclude that the two compositions $\mathcal{L}^k\circ\pi^*$ 
and $\pi^*\circ\mathcal{L}^k$ are always defined. That they are inverses of each other, is obvious from the construction of $\mathcal{L}^k$.

It remains to check that $\mathcal{L}^k$ is smooth. Let $q$ be a plot of $\Omega_f^k(X_1)\oplus_{comp}\Omega^k(X_2)$, and let $U$ be its domain of definition. Then for all $u\in U$ we have that 
$q(u)=q(u)_1\oplus q(u)_2$ for some $q(u)_1\in\Omega_f^k(X_1)$ and $q(u)_2\in\Omega^k(X_2)$, and the assignments $u\mapsto q(u)_1$ and $u\mapsto q(u)_2$ are plots of $\Omega_f^k(X_1)$ and 
of $\Omega^k(X_2)$ respectively. 

To show that $u\mapsto\mathcal{L}^k(q(u)_1\oplus q(u)_2)$ is a plot of $\Omega^k(X_1\cup_f X_2)$, as is required for showing the smoothness of $\mathcal{L}^k$, we need to consider a plot 
$p:U'\to X_1\cup_f X_2$ and show that the evaluation map $(u,u')\mapsto\mathcal{L}^k(q(u)_1\oplus q(u)_2)(p)(u')$ is a usual smooth section of $\Lambda^k(U\times U')$. It suffices to assume that $U'$ is 
connected; then $p$ lifts to either a plot $p_1$ of $X_1$ or to a plot $p_2$ of $X_2$. Depending on these two cases, the evaluation map for $q$ either has form $(u,u')\mapsto q(u)_1(p_1)(u')$ or 
$(u,u')\mapsto q(u)_2(p_2)(u')$, which in both cases is a smooth section of $\Lambda^k(U\times U')$, because $q(u)_1,q(u)_2$ are plots, whence the claim.
\end{proof}

Theorem \ref{cal:L:is:inverse:of:pullback:pi:thm} trivially implies the following.

\begin{cor}\label{omega:of:glued:cor}
The map $\pi^*$ is a diffeomorphism $\Omega^k(X_1\cup_f X_2)\to\Omega_f^k(X_1)\oplus_{comp}\Omega^k(X_2)$. 
\end{cor}

\subsection{The differential and gluing}

We shall consider next the behavior of the differential (the coboundary) $d$ operator under gluing. Let $X_1$ and $X_2$ be two diffeological spaces, and let $f:X_1\supseteq Y\to X_2$ be a smooth 
map. For every $\omega\in\Omega^k(X_1\cup_f X_2)$ the differential $d\omega\in\Omega^{k+1}(X_1\cup_f X_2)$ is determined by the collection of the usual differentials of standard $k$-forms $\omega(p)$ 
for all plots $p$ of $X_1\cup_f X_2$. Now, we have just seen that $\omega$ is essentially the union (or the wedge) of a $k$-form on $X_1$ with a $k$-form on $X_2$, and every plot $p$ of $X_1\cup_f X_2$ 
is in some sense a union of a plot of $X_1$ with a plot of $X_2$ (one of which could be absent if the domain of definition of $p$ is connected), see \cite{pseudobundle} and Lemma 4.1 in 
\cite{pseudobundle-pseudometrics}. The following therefore is an expected statement.

\begin{thm}\label{differential:and:gluing:thm}
Let $X_1$ and $X_2$ be two diffeological spaces, let $f:X_1\supseteq Y\to X_2$ be a smooth map, and let $\omega\in\Omega^k(X_1\cup_f X_2)$ be a $k$-form. Let $\pi^*(\omega)=\omega_1\oplus\omega_2$. 
Then 
$$\pi^*(d\omega)=d\omega_1\oplus d\omega_2.$$
\end{thm}

\begin{proof}
Let $p:U\to X_1\sqcup X_2$ be a plot of $X_1\sqcup X_2$. We need to compare $\pi^*(d\omega)(p)$ with $(d\omega_1\oplus d\omega_2)(p)$. It suffices to assume that $U$ is connected; then $p$ essentially 
coincides with either a plot $p_1$ of $X_1$ or a plot $p_2$ of $X_2$. Suppose it coincides with $p_1$. Then by construction and definition 
$$\pi^*(d\omega)(p)=d\omega(\pi\circ p)=d(\omega(\pi\circ p))=d(\omega_1(p_1)),$$
$$(d\omega_1\oplus d\omega_2)(p)=(d\omega_1)(p_1)=d(\omega_1(p_1)),$$ so the desired equality is true. Since the case when $p$ is equivalent to a plot of $X_2$ is completely analogous, we obtain 
the desired claim.
\end{proof}

\subsection{The extendibility conditions $\calD_1^{\Omega^k}=\calD_2^{\Omega^k}$ and the images of $\pi_1^{\Omega},\pi_2^{\Omega}$}

So far we have only assumed that the gluing map $f$ is smooth (which is always required for the gluing construction). Obtaining further claims needs some additional conditions, that we call extendibility 
conditions and describe in this section. 

\begin{defn}
Let $X_1$ and $X_2$ be two diffeological spaces, let $f:X_1\supseteq Y\to X_2$ be a smooth map, and let $i:Y\hookrightarrow X_1$ and $j:f(Y)\hookrightarrow X_2$ be the natural inclusions. We say that 
$f$ satisfies the \textbf{$k$-th extendibility condition} if 
$$i^*(\Omega^k(X_1))=(f^*j^*)(\Omega^k(X_2)).$$ Denote now by $\calD_1^{\Omega^k}$ the diffeology on $\Omega^k(Y)$ that is the pushforward of the diffeology of $\Omega^k(X_1)$ by the map $i^*$; 
likewise, denote by $\calD_2^{\Omega^k}$ the diffeology on $\Omega^k(Y)$ that is the pushforward of the diffeology of $\Omega^k(X_2)$ by the map $f^*j^*$. We say that $f$ satisfies the \textbf{$k$-th 
smooth extendibility condition} if 
$$\calD_1^{\Omega^k}=\calD_2^{\Omega^k}.$$
\end{defn}

The need for these two conditions is based on the following lemma and is rendered explicit by the corollary that follows it.

\begin{lemma}\label{when:two:forms:are:compatible:lem}
Let $\omega_1\in\Omega^k(X_1)$ and $\omega_2\in\Omega^k(X_2)$ be two $k$-forms, and let $f:X_1\supseteq Y\to X_2$ be a smooth map. The forms $\omega_1$ and $\omega_2$ are compatible if 
and only if 
$$i^*\omega_1=(f^*j^*)\omega_2.$$
\end{lemma}

\begin{proof}
Let $\omega_1$ and $\omega_2$ be compatible, and let $p:U\to Y\subseteq X_1$ be an arbitrary plot of $Y$. Since $(i^*\omega_1)(p)=\omega_1(i\circ p)=\omega_1(p)$ and 
$(f^*j^*)(\omega_2)(p)=\omega_2(f\circ j\circ p)=\omega_2(f\circ p)$, we obtain the desired equality $i^*\omega_1=(f^*j^*)\omega_2$ by the assumption of compatibility of $\omega_1$ and $\omega_2$. 

Suppose now that $i^*\omega_1=(f^*j^*)\omega_2$ holds; let us show that $\omega_1$ and $\omega_2$ are compatible. Let again $p$ be any plot of $Y$. Then 
$$\omega_1(p)=\omega_1(i\circ p)=(i^*\omega_1)(p)=(f^*j^*)(\omega_2)(p)=\omega_2(f\circ j\circ p)=\omega_2(f\circ p),$$ therefore the compatibility condition $\omega_1(p)=\omega_2(f\circ p)$ follows 
from the assumption.
\end{proof}

\begin{cor}\label{when:pi-omega:are:surjective:cor}
Let $\pi_1^{\Omega}:\Omega^k(X_1)\oplus_{comp}\Omega^k(X_2)\to\Omega^k(X_1)$ and $\pi_2^{\Omega}:\Omega^k(X_1)\oplus_{comp}\Omega^k(X_2)\to\Omega^k(X_2)$ be induced by the standard 
direct sum projections. Then $\pi_1^{\Omega}$ and $\pi_2^{\Omega}$ are both surjective if and only if $f$ satisfies the extendibility condition $i^*(\Omega^k(X_1))=(f^*j^*)(\Omega^k(X_2))$.
\end{cor}

\begin{proof}
A form $\omega_1\in\Omega^k(X_1)$ belongs to the range of $\pi_1^{\Omega}$ if and only if there exists a form $\omega_2\in\Omega^k(X_2)$ such that $\omega_1$ and $\omega_2$ are compatible. 
By Lemma \ref{when:two:forms:are:compatible:lem} this is equivalent to $i^*\omega_1\in(f^*j^*)(\Omega^k(X_2))$. Asking for this being true for all $\omega_1\in\Omega^k(X_1)$ is obviously equivalent 
to the inclusion $i^*(\Omega^k(X_1))\subseteq(f^*j^*)(\Omega^k(X_2))$. Applying exactly the same reasoning to an arbitrary $\omega_2\in\Omega^k(X_2)$, we obtain the claim.
\end{proof}

\begin{rem}
As is clear from the proof of Corollary \ref{when:pi-omega:are:surjective:cor}, the necessary and sufficient condition for only $\pi_1^{\Omega}$ to be surjective is 
$i^*(\Omega^k(X_1))\subseteq(f^*j^*)(\Omega^k(X_2))$; that for surjectivity of only $\pi_2^{\Omega}$ is $(f^*j^*)(\Omega^k(X_2))\subseteq i^*(\Omega^k(X_1))$.
\end{rem}

\subsection{The De Rham groups $H_{dR}^k(X_1\cup_f X_2)$}

We shall now consider the De Rham groups of $X_1\cup_f X_2$ as they relate to those of $X_1$ and $X_2$. Their description is based on the straightforward behavior of the differential under gluing 
(Theorem \ref{differential:and:gluing:thm}). 

\paragraph{Cocycles and coboundaries} Some observations regarding the complex of the coccyges, and that of the coboundaries, are immediate from Theorem \ref{differential:and:gluing:thm}. 

\begin{lemma}\label{cocycles:and:coboundaries:in:glued:lem}
Let $X_1$ and $X_2$ be two diffeological spaces, and let $f:X_1\supseteq Y\to X_2$ be a diffeomorphism such that $\calD_1^{\Omega^k}=\calD_2^{\Omega^k}$. Then:
$$Z_{dR}^k(X_1\cup_f X_2)\subseteq Z_{dR}^k(X_1)\oplus Z_{dR}^k(X_2),$$
$$B_{dR}^k(X_1\cup_f X_2)\subseteq B_{dR}^k(X_1)\oplus B_{dR}^k(X_2).$$
\end{lemma}

\begin{proof}
This follows from Theorem \ref{differential:and:gluing:thm}, whose essence is that $d\omega$, for any $\omega\in\Omega^k(X_1\cup_f )$, is canonically identified, via an isomorphism, to 
$d\omega_1\oplus d\omega_2$. It is then obvious that $B_{dR}^k(X_1\cup_f X_2)\subseteq B_{dR}^k(X_1)\oplus B_{dR}^k(X_2)$. Furthermore, $d\omega=0$ if and only if both $d\omega_1=0$ and 
$d\omega_2=0$, therefore $Z_{dR}^k(X_1\cup_f X_2)\subseteq Z_{dR}^k(X_1)\oplus Z_{dR}^k(X_2)$. 
\end{proof}

\paragraph{Compatibility of $d\omega_1$ and $d\omega$ \emph{vs.} compatibility of $\omega_1$ and $\omega_2$} That the latter implies the former, is implicit in Theorem \ref{differential:and:gluing:thm}. 
We shall now discuss why the former implies the latter.

\begin{lemma}\label{cocycles:and:coboundaries:split:lem}
The differentials $d\omega_1$ and $d\omega_2$ of two forms $\omega_1\in\Omega^{k-1}(X_1)$ and $\omega_2\in\Omega^{k-1}(X_2)$ are compatible if and only if the forms $\omega_1$ and $\omega_2$ 
are themselves compatible. In particular, 
$$Z_{dR}^k(X_1\cup_f X_2)=Z_{dR}^k(X_1)\oplus Z_{dR}^k(X_2)$$
$$B_{dR}^k(X_1\cup_f X_2)=B_{dR}^k(X_1)\oplus B_{dR}^k(X_2).$$
\end{lemma}

\begin{proof}
Let $p:\matR^n\supseteq U\to Y\subseteq X_1$ be a plot, and let $\omega_1\in\Omega^{k-1}(X_1)$ and $\omega_2\in\Omega^{k-1}(X_2)$ be two forms such that $d\omega_1$ and $d\omega_2$ are 
compatible. Thus, $(d\omega_1)(p)=(d\omega_2)(f\circ p)$, that is, $d(\omega_1(p))=d(\omega_2(f\circ p))$, where $\omega_1':=\omega_1(p)$ and $\omega_2':=\omega_2(f\circ p)$ are two usual 
differential forms in $\Omega^{k-1}(U)$. Furthermore, they are such that $\omega_1'-\omega_2'$ is a cocycle, hence its defines an element of $H^{k-1}(U)$. If $U$ is simply connected, $H^{k-1}(U)$ is trivial, so 
$\omega_1(p)=\omega_2(f\circ p)$. It remains to recall the locality property for diffeological differential forms (\cite{iglesiasBook}, Section 6.36) to conclude that $\omega_1(p)=\omega_2(f\circ p)$ for all other 
plots $p$ of $Y$.

Thus, if $d\omega_1$ and $d\omega_2$ are compatible, which includes the case when they are both zero, then $\mathcal{L}^{k-1}(\omega_1\oplus\omega_2)$ is well-defined. Since
$d\circ\mathcal{L}^{k-1}=\mathcal{L}^k\circ(d\oplus d)$, we obtain the claim.
\end{proof}

\paragraph{The diffeomorphism $H_{dR}^k(X_1\cup_f X_2)\cong H_{dR}^k(X_1)\oplus H_{dR}^k(X_2)$} The following is now a trivial consequence of Lemma \ref{cocycles:and:coboundaries:split:lem}.

\begin{thm}\label{de:rham:cohomology:splits:thm}
Let $X_1$ and $X_2$ be two diffeological spaces, and let $f:X_1\supseteq Y\to X_2$ be a diffeomorphism such that $\calD_1^{\Omega^k}=\calD_2^{\Omega^k}$ for all $k$. Then 
$$H_{dR}^k(X_1)\oplus H_{dR}^k(X_2)\cong H_{dR}^k(X_1\cup_f X_2)$$ via the isomorphism induced by the chain map $\{\mathcal{L}^k\}$.
\end{thm}

\subsection{The pseudo-bundles $\Lambda^k(X_1\cup_f X_2)$ relative to $\Lambda^k(X_1)$ and $\Lambda^k(X_2)$}

We now consider the pseudo-bundles $\Lambda^k(X_1\cup_f X_2)$ (see \cite{forms-gluing} for the case of $k=1$, which is treated in a somewhat more general manner). We only do so under substantial 
restrictions on $f$. The first of them is that $f$ be a diffeomorphism of its domain with its image, and this is necessary for us (we do not know yet how to treat a more general case); the second restriction is that 
$f$ satisfy the $k$-th smooth extendibility condition, and this, in some cases, may not be strictly necessary (but the results would get far more cumbersome with it). Notice that due to the assumption that $f$ is 
a diffeomorphism, the map $\tilde{i}_1$ is invertible, and $\Omega_f^k(X_1)=\Omega^k(X_1)$, that is, every $k$-form on $X_1$ is $f$-invariant.

\subsubsection{The vanishing of forms in $\Omega^k(X_1\cup_f X_2)$}

Recall that each fibre of $\Lambda^k(X_1\cup_f X_2)$ is the quotient of form $\Omega^k(X_1\cup_f X_2)/\Omega_x^k(X_1\cup_f X_2)$.

\begin{thm}\label{vanishing:forms:according:to:point:thm}
Let $X_1$ and $X_2$ be two diffeological spaces, let $f:X_1\supseteq Y\to X_2$ be a diffeomorphism satisfying the $k$-th smooth compatibility condition $\calD_1^{\Omega^k}=\calD_2^{\Omega^k}$, and 
let $x\in X_1\cup_f X_2$ be a point. The the space $\Omega_x^k(X_1\cup_f X_2)$ of $k$-forms vanishing at $x$ is defined by the following:
$$\pi^*(\Omega_x^k(X_1\cup_f X_2))\cong\left\{\begin{array}{ll} 
\Omega_{\tilde{i}_1^{-1}(x)}^k(X_1)\oplus_{comp}\Omega^k(X_2) & \mbox{if }x\in i_1(X_1\setminus Y), \\
\Omega_{\tilde{i}_1^{-1}(x)}^k(X_1)\oplus_{comp}\Omega_{i_2^{-1}(x)}^k(X_2) & \mbox{if }x\in i_2(f(Y)), \\
\Omega^k(X_1)\oplus_{comp}\Omega_{i_2^{-1}(x)}^k(X_2) & \mbox{if }x\in i_2(X_2\setminus f(Y)).
\end{array}\right.$$
\end{thm}

\begin{proof}
Let first $\omega\in\Omega_x^k(X_1\cup_f X_2)$, and let $\pi^*\omega$ be written as $\omega_1\oplus\omega_2$. If $x\in i_1(X_1\setminus Y)$, we need to show that $\omega_1$ vanishes at 
$\tilde{i}_1^{-1}(x)$. Let $p_1$ be a plot of $X_1$, with connected domain of definition, such that $p_1(0)=\tilde{i}_1^{-1}(x)$. Then $p:=\tilde{i}_1\circ p_1$ is a plot of $X_1\cup_f X_2$ such that $p(0)=x$. 
We have by construction $\omega(p)=\omega_1(p_1)$, therefore $\omega_1(p_1)(0)=\omega(p)(0)=0$, therefore $\omega_1$ vanishes at $\tilde{i}_1^{-1}(x)$. This proves that 
$$\pi^*(\Omega_x^k(X_1\cup_f X_2))\subseteq\Omega_{\tilde{i}_1^{-1}(x)}^k(X_1)\oplus_{comp}\Omega^k(X_2).$$ The proof that 
$$\pi^*(\Omega_x^k(X_1\cup_f X_2))\subseteq\Omega^k(X_1)\oplus_{comp}\Omega_{i_2^{-1}(x)}^k(X_2)$$ is completely analogous. 

Let thus $x\in i_2(f(Y))$. If $p_2$ is a plot of $X_2$ such that $p_2(0)=i_2^{-1}(x)$, we have, as before, $\omega_2(p_2)=\omega(i_2\circ p_2)$, and $i_2(p_2(0))=x$, so $\omega_2$ vanishes at $i_2^{-1}(x)$. 
Let $p_1$ be a plot of $X_1$. Again, $\omega_1(p_1)=\omega(\tilde{i}_1\circ p_1)$ and $(\tilde{i}_1\circ p_1)(0)=x$, so $\omega_1$ vanishes at $\tilde{i}_1^{-1}(x)$. Therefore 
$$\pi^*(\Omega_x^k(X_1\cup_f X_2))\subseteq\Omega_{\tilde{i}_1^{-1}(x)}^k(X_1)\oplus_{comp}\Omega_{i_2^{-1}(x)}^k(X_2).$$

Let us establish the reverse inclusion. Let $\omega_1\in\Omega_{x_1}^k(X_1)$ and $\omega_2\in\Omega^k(X_2)$ be two compatible forms, and let $\omega=\mathcal{L}^k(\omega_1\oplus\omega_2)$. 
Let $p$ be a plot of $X_1\cup_f X_2$ with connected domain of definition and such that $p(0)=\tilde{i}_1(x_1)=:x$. Then $p_1:=\tilde{i}_1^{-1}\circ p$ is a plot of $X_1$ and $p_1(0)=x_1$. Furthermore, 
$\omega_1(p_1)=\omega(p)$ by construction. We thus conclude that $\omega(p)(0)=0$, hence $\omega$ vanishes at $x$, and in particular, we obtain the first claim. Analogously, if 
$\omega_1\in\Omega^k(X_1)$ and $\omega_2\in\Omega_{x_2}^k(X_2)$ are compatible then $\omega=\mathcal{L}^k(\omega_1\oplus\omega_2)$ vanishes at $i_2(X_2)$; this yields the third claim. 
Finally, since 
$$\Omega_{x_1}^k(X_1)\oplus_{comp}\Omega_{f(x_1)}^k(X_2)=\left(\Omega_{x_1}^k(X_1)\oplus_{comp}\Omega^k(X_2)\right)\,\cap\,\left(\Omega^k(X_1)\oplus_{comp}\Omega_{f(x_1)}^k(X_2)\right)$$ 
for any $x\in Y$, we obtain the second claim, and the proof is finished.
\end{proof}

\subsubsection{The fibres of $\Lambda^k(X_1\cup_f X_2)$} 

We first define an appropriate compatibility notion for elements of fibres of form $\Lambda_x^k(X_1)$ and $\Lambda_{f(x)}^k(X_2)$, for $x\in Y$.

\begin{defn}\label{compatible:elements:of:lambda:defn}
Let $x\in Y$, let $\alpha_1=\omega_1+\Omega_x^k(X_1)\in\Lambda_x^k(X_1)$, and let $\alpha_2=\omega_2+\Omega_{f(x)}^k(X_2)\in\Lambda_{f(x)}^k(X_2)$. We say that $\alpha_1$ and $\alpha_2$ are 
\textbf{compatible} if any two forms $\omega_1'\in\alpha_1\subseteq\Omega^k(X_1)$ and $\omega_2'\in\alpha_2\subseteq\Omega^k(X_2)$ are compatible.
\end{defn}

We denote 
$$\Lambda_x^k(X_1)\oplus_{comp}\Lambda_{f(x)}^k(X_2):=\{\alpha_1\oplus\alpha_2\,|\,\alpha_1\mbox{ and }\alpha_2\mbox{ are compatible}\}$$ for every $x\in Y$. 

\begin{thm}\label{fibres:of:lambda:thm}
Let $X_1$ and $X_2$ be two diffeological spaces, let $f:X_1\supseteq Y\to X_2$ be a diffeomorphism such that $\calD_1^{\Omega^k}=\calD_2^{\Omega^k}$, and let $x\in X_1\cup_f X_2$. Then:
$$\Lambda_x^k(X_1\cup_f X_2)\cong\left\{\begin{array}{ll} 
\Lambda_{\tilde{i}_1^{-1}(x)}^k(X_1) & \mbox{if }x\in i_1(X_1\setminus Y), \\
\Lambda_{\tilde{i}_1^{-1}}^k(X_1)\oplus_{comp}\Lambda_{i_2^{-1}(x)}^k(X_2) & \mbox{if }x\in i_2(f(Y)), \\
\Lambda_{i_2^{-1}(x)}^k(X_2) & \mbox{if }x\in i_2(X_2\setminus f(Y)).
\end{array}\right.$$
\end{thm}

\begin{proof}
This is a simple consequence of Theorem \ref{vanishing:forms:according:to:point:thm}. It amounts to checking that 
$$\left(\Omega^k(X_1)\oplus\Omega^k(X_2)\right)/\left(\Omega_{\tilde{i}_1^{-1}(x)}^k(X_1)\oplus_{comp}\Omega^k(X_2)\right)\cong\Omega^k(X_1)/ \Omega_{\tilde{i}_1^{-1}(x)}^k(X_1),$$
\begin{flushleft}
$\left(\Omega^k(X_1)\oplus\Omega^k(X_2)\right)/\left(\Omega_{\tilde{i}_1^{-1}(x)}^k(X_1)\oplus_{comp}\Omega_{i_2^{-1}(x)}^k(X_2)\right)\cong$
\end{flushleft}
\begin{flushright}
$\cong\left(\Omega^k(X_1)/ \Omega_{\tilde{i}_1^{-1}(x)}^k(X_1)\right)\oplus_{comp}\left(\Omega^k(X_2)/ \Omega_{i_2^{-1}(x)}^k(X_2)\right),$
\end{flushright}
$$\left(\Omega^k(X_1)\oplus\Omega^k(X_2)\right)/\left(\Omega^k(X_1)\oplus_{comp}\Omega_{i_2^{-1}(x)}^k(X_2)\right)\cong\Omega^k(X_2)/ \Omega_{i_2^{-1}(x)}^k(X_2),$$ and this is done by 
completely standard reasoning, of which we omit the details.
\end{proof}

\subsubsection{The characteristic maps $\tilde{\rho}_1^{\Lambda^k}$ and $\tilde{\rho}_2^{\Lambda^k}$}

It is worth noting that under the assumption of the gluing map $f$ being a diffeomorphism such that $\calD_1^{\Omega^k}=\calD_2^{\Omega^k}$, the total space $\Lambda^k(X_1\cup_f X_2)$ is a span of the 
total spaces $\Lambda^k(X_1)$ and $\Lambda^k(X_2)$: it admits two (surjective partially defined) maps 
$$\tilde{\rho}_1^{\Lambda^k}:(\pi^{\Lambda^k})^{-1}(\tilde{i}_1(X_1))\to\Lambda^k(X_1),\,\,\,\tilde{\rho}_2^{\Lambda^k}:(\pi^{\Lambda^k})^{-1}(i_2(X_2))\to\Lambda^k(X_2),$$ where $\pi^{\Lambda^k}$ 
is the pseudo-bundle projection $\Lambda^k(X_1\cup_f X_2)\to X_1\cup_f X_2$. 

The maps $\tilde{\rho}_1^{\Lambda^k}$ and $\tilde{\rho}_2^{\Lambda^k}$ are induced by the pullback maps $\tilde{i}_1^*$ and $i_2^*$ respectively, and can also be given a more direct description, by 
representing $\Lambda^k(X_1\cup_f X_2)$ as a quotient of 
$$(X_1\sqcup X_2)\times(\Omega^k(X_1)\oplus_{comp}\Omega^k(X_2))=X_1\times(\Omega^k(X_1)\oplus_{comp}\Omega^k(X_2))\,\sqcup\,X_2\times(\Omega^k(X_1)\oplus_{comp}\Omega^k(X_2)).$$ 
The domain of definition of $\tilde{\rho}_1^{\Lambda^k}$ corresponds to $X_1\times(\Omega^k(X_1)\oplus_{comp}\Omega^k(X_2))$, and $\tilde{\rho}_1^{\Lambda^k}$ itself is induced by the projection of 
$X_1\times(\Omega^k(X_1)\oplus_{comp}\Omega^k(X_2))$ to $X_1\times\Omega^k(X_1)$. The direct construction of $\tilde{\rho}_2^{\Lambda^k}$ is completely analogous.

Both of these maps are smooth and linear by construction. Furthermore, the following is true.

\begin{prop}\label{tilde-rho:are:subductions:prop}
Let $X_1$ and $X_2$ be two diffeological spaces, and let $f:X_1\supseteq Y\to X_2$ be a diffeomorphism such that $\calD_1^{\Omega^k}=\calD_2^{\Omega^k}$. Then the maps 
$$\tilde{\rho}_1^{\Lambda^k}:(\pi^{\Lambda^k})^{-1}(\tilde{i}_1(X_1))\to\Lambda^k(X_1),\,\,\,\tilde{\rho}_2^{\Lambda^k}:(\pi^{\Lambda^k})^{-1}(i_2(X_2))\to\Lambda^k(X_2),$$ where 
$(\pi^{\Lambda^k})^{-1}(\tilde{i}_1(X_1))$ and $(\pi^{\Lambda^k})^{-1}(i_2(X_2))$ are considered with the subset diffeologies relative to their inclusions in $\Lambda^k(X_1\cup_f X_2)$, are subductions. 
\end{prop}

\begin{proof}
The two cases of $\tilde{\rho}_1^{\Lambda^k}$ and $\tilde{\rho}_2^{\Lambda^k}$ are fully analogous, so we only consider the first of them. Let $q_1:U\to\Lambda^k(X_1)$ be a plot of $\Lambda^k(X_1)$ 
(possibly a constant one). We need to show that (at least up to restricting $U$) there exists a plot $q:U\to\Lambda^k(X_1\cup_f X_2)$ such that $q_1=\tilde{\rho}_1^{\Lambda^k}\circ q$. 

By definition of the diffeology of any $\Lambda^k(\cdot)$, there exists a (local) lift $p_1:U\to X_1\times\Omega^k(X_1)$ of $q_1$, of form $p_1(u)=(\pi_1^{\Lambda^k}(q_1(u)),p_1^{\Omega^k}(u))$ for $u\in U$, 
where $p_1^{\Omega^k}:U\to\Omega^k(X_1)$ is a plot of $\Omega^k(X_1)$. By the smooth compatibility condition, there exists a plot $p_2^{\Omega^k}:U\to\Omega^k(X_2)$ of $\Omega^k(X_2)$ such that 
$$i^*\circ p_1^{\Omega^k}=(f^*j^*)\circ p_2^{\Omega^k}.$$ By Lemma \ref{when:two:forms:are:compatible:lem} this means that $p_1^{\Omega^k}(u)$ and $p_2^{\Omega^k}(u)$ are compatible for all $u\in U$. 
Therefore $p:U\to(X_1\cup_f X_2)\times(\Omega^k(X_1)\oplus_{comp}\Omega^k(X_2))$ given by 
$$p(u)=(\tilde{i}_1(\pi_1^{\Lambda^k}(q_1(u))),p_1^{\Omega^k}(u)\oplus p_2^{\Omega^k}(u))$$ is well-defined, and by construction it is a plot of 
$(X_1\cup_f X_2)\times(\Omega^k(X_1)\oplus_{comp}\Omega^k(X_2))$. Therefore its composition $q=\pi^{\Omega,\Lambda}\circ p$ with the defining projection $\pi^{\Omega,\Lambda}$ of 
$\Lambda^k(X_1\cup_f X_2)$ is a plot of $\Lambda^k(X_1\cup_f X_2)$, and by construction $\tilde{\rho}_1^{\Lambda^k}\circ q=q_1$, which completes the proof.  
\end{proof}

The proof of Proposition \ref{tilde-rho:are:subductions:prop} provides a working characterization of the diffeology of $\Lambda^k(X_1\cup_f X_2)$, even without any additional conditions on the gluing map $f$. 
Namely, any plot of $\Lambda^k(X_1\cup_f X_2)$ locally has a lift of form $u\mapsto(p(u),p_1^{\Omega^k}(u)\oplus p_2^{\Omega^k}(u))$, where $p$ is any plot of $X_1\cup_f X_2$, and $p_1^{\Omega^k}$ 
and $p_2^{\Omega^k}$ are any two plots of $\Omega^k(X_1)$ and $\Omega^k(X_2)$ respectively such that $i^*\circ p_1^{\Omega^k}=f^*\circ j^*\circ p_2^{\Omega^k}$.

\section{The operator $d+d^*$ in general is not defined}

In this section we examine the ingredients that usually go into the construction of the De Rham operator as the operator $d+d^*$, showing (via examples based on the gluing construction) that they do not 
extend, in any straightforward manner, to the diffeological context; whenever, as in the case of volume forms, a formally defined extension exists, it is not really suitable for the purpose it is meant to achieve.

\subsection{The differential is not well-defined as a map on $\Lambda^k(X)$}

Let $\alpha\in\Lambda^k(X)$, and let $x:=\pi^{\Lambda^k}(\alpha)$. \emph{A priori}, if a form $\omega\in\Omega^k(X)$ vanishes at $x$, it is not clear why its differential should vanish at $x$ as well; this 
condition would be needed to ensure that the differential on $\Lambda^k(X)$ could be defined by $d(\omega+\Omega_x^k(X))=d\omega+\Omega_x^{k+1}(X)$. However, already the case of $k=0$ illustrates 
that this cannot be done. It suffices to consider, on the standard $\matR$, any smooth function $h$ such that $h(0)=0$ and $h'(0)\neq 0$ (for instance, $h(x)=\sin(x)$).

\subsection{The dimension of a diffeological space and pseudo-bundles $\Lambda^k(X)$}

Although there exists a notion of dimension for diffeological spaces that is similar to the standard one, its implications for the dimensions of fibres of $\Lambda^k(X)$ are not entirely similar to those in the 
standard case. Specifically, if $\dim(X)=n$ then all pseudo-bundles $\Lambda^{n+k}(X)$, $k=1,2,\ldots$, are trivial; but the dimensions of $\Lambda^k(X)$ with $k=1,\ldots,n$ are not bounded by $n$ and 
can in fact be arbitrarily large.

\subsubsection{The dimension of $X_1\cup_f X_2$}

The dimension of a diffeological space is an extension of the usual notion. It is based on the fact that, although the diffeology $\calD$ of any given diffeological space $X$ can be quite large, it is usually 
determined by a smaller subset $\mathcal{A}$ of it, called a \textbf{generating family} of $\calD$. More specifically, a subset $\mathcal{A}\subseteq\calD$ is called a generating family of $\calD$ if for any 
plot $p:U\to X$ in $\calD$ and for any $u\in U$ there exists a neighborhood $U'\subseteq U$ of $u$ such that either $p|_{U'}$ is constant or there exists a plot $q:U''\to X$ in $\mathcal{A}$ and an ordinary 
smooth map $h:U'\to U''$ such that $p|_{U'}=q\circ h$. We can re-state this briefly by saying that locally every $p\in\calD$ either is constant or filters through a plot in $\mathcal{A}$. Almost always, a 
diffeology admits many generating families.

\begin{defn}
Let $X$ be a diffeological space, and let $\calD$ be its diffeology. The \textbf{dimension} of any generating family $\mathcal{A}=\{q_{\alpha}:U_{\alpha}\to X\}_{\alpha}$ is the supremum of the dimensions of 
the domains of definition of all $q_{\alpha}\in\mathcal{A}$, 
$$\dim(\mathcal{A})=\mbox{sup}\{\dim(U_{\alpha})\}.$$ If no supremum exists, the dimension is said to be infinite. The \textbf{dimension of $X$} is the infimum of the dimensions of all generating families of 
$\calD$, 
$$\dim(X)=\mbox{inf}\{\dim(\mathcal{A})\,|\,\mathcal{A}\subseteq\calD\mbox{ is a generating family of }\calD\}.$$ If $\calD$ has no generating family with finite dimension, $X$ is said to have infinite dimension.
\end{defn}

The following is then a trivial observation.

\begin{lemma}\label{dimension:of:glued:lem}
Let $X_1$ and $X_2$ be two diffeological spaces of finite dimensions, and let $f:X_1\supseteq Y\to X_2$ be a smooth map. Then 
$$\dim(X_1\cup_f X_2)=\mbox{max}\{\dim(X_1),\dim(X_2)\}\,\,\mbox{ if }Y\neq X_1,\,\,\mbox{ and }\dim(X_1\cup_f X_2)=\dim(X_2)\mbox{ otherwise}.$$ In particular, $X_1\cup_f X_2$ has finite dimension if 
and only if both $X_1$ and $X_2$ have finite dimension.
\end{lemma}

\begin{proof}
Let $\mathcal{A}$ be a generating family of the gluing diffeology on $X_1\cup_f X_2$. We can assume that all plots in $\mathcal{A}$ have connected domains of definition. If $Y=X_1$ then 
$X_1\cup_f X_2\cong X_2$, so the second statement is obvious. Assume that $Y$ is properly contained in $X_1$. Let $\mathcal{A}_1\subseteq\mathcal{A}$ be the subset of all plots of $\mathcal{A}$ that 
have lifts to plots of $X_1$; let $\mathcal{A}_2\subseteq\mathcal{A}$ be the subset of plots with lifts to $X_2$. Then $\mathcal{A}=\mathcal{A}_1\cup\mathcal{A}_2$, and $\mathcal{A}_1$ and $\mathcal{A}_2$ 
are in a natural correspondence with specific generating families $\mathcal{A}_1'$ and $\mathcal{A}_2'$ of the diffeologies of $X_1$ and $X_2$ respectively, and since $X_1\setminus Y$ is non-empty, 
$\mathcal{A}\setminus\mathcal{A}_2$ is non-empty as well. Therefore we have the inequality $\dim(X_1\cup_f X_2)\leqslant\mbox{max}\{\dim(X_1),\dim(X_2)\}$.

\emph{Vice versa}, any two generating families of the diffeologies on $X_1$ and $X_2$ yield automatically a generating family for the gluing diffeology on $X_1\cup_f X_2$. Therefore we obtain the reverse 
inequality, and so the final claim. 
\end{proof}

\subsubsection{The dimension of $X$ and pseudo-bundles $\Lambda^k(X)$}

For any diffeological space $X$ and for any differential form $\omega\in\Omega^k(X)$, there is a standard way to associate to $\omega$ a smooth section of $\Lambda^k(X)$. This section is defined as the 
assignment
$$x\mapsto\pi^{\Omega^k,\Lambda^k}(x,\omega),$$ where, recall, $\pi^{\Omega^k,\Lambda^k}:X\times\Omega^k(X)\to\Lambda^k(X)$ is the defining quotient projection of $\Lambda^k(X)$ (this is the 
tautological $k$-form corresponding to $\omega$, that is mentioned in \cite{iglesiasBook}, p. 160). The following is a known fact (see \cite{iglesiasBook}, Section 6.37), but for completeness we provide 
a proof.

\begin{lemma}
Let $X$ be a diffeological space of finite dimension $n$. Then $\Omega^k(X)$ is trivial for $k>n$.  
\end{lemma}

\begin{proof}
Choose a fixed $k>n$. Let $\mathcal{A}$ be a generating family of plots of the diffeology of $X$ that has dimension $n$ (that is, every plot in $\mathcal{A}$ is defined on a domain in $\matR^m$ with 
$m\leqslant n$, and at least one plot is defined on a domain in $\matR^n$) and let $\omega\in\Omega^k(X)$ be a form. We need to show that $\omega$ is the zero form. Let first $p\in\mathcal{A}$; by 
assumption, the (usual) dimension of its domain of definition is strictly less than $k$. Therefore obviously $\omega(p)=0$. Let now $q:U\to X$ be any random plot of $X$. Then for every $u\in U$ there exists 
a subdomain $U'\subseteq U$ such that $q|_{U'}=p\circ F$ for some ordinary smooth map $F:U'\to U''$ and for some plot $p:U''\to X$ that belongs to $\mathcal{A}$. Therefore 
$\omega(q|_{U'})=F^*(\omega(p))=0$. Since this is true for any $u\in U$, we conclude that $\omega(q)=0$, whence the claim. 
\end{proof}

The following is then immediately obvious.

\begin{cor}
If $X$ is a diffeological space of dimension $n$ then all pseudo-bundles $\Lambda^k(X)$ for $k>n$ are trivial.
\end{cor}

Suppose now that there exists a volume form $\omega\in\Omega^n(X)$ on $X$ of dimension $n$. Let $\mathcal{A}$ be a generating family for the diffeology of $X$ that has dimension $n$; let 
$\mathcal{A}_n\subseteq\mathcal{A}$ be the subset consisting of precisely the plots in $\mathcal{A}$ whose domain of definition has dimension $n$. Obviously, if $p\in\mathcal{A}\setminus\mathcal{A}_n$ 
then $\omega(p)=0$. On the other hand, there are diffeological spaces such that $\mathcal{A}_n$ contains at least two plots that are not related by a smooth substitution, which implies that the dimension, 
in the sense of pseudo-bundles, of $\Lambda^n(X)$ can be greater than $n$. In fact, it can be arbitrarily greater, as the following example shows.

\begin{example}\label{dim:X:and:dim:lambda:are:different:ex}
Let $m\in\matN$, $m\geqslant 2$, be any, and let $X$ be the wedge at the origin of $m$ copies of $\matR$ (each copy endowed with its standard diffeology), endowed with the corresponding gluing diffeology. 
It is quite clear that $X$ is finite-dimensional, and that its dimension is equal to $1$. However, applying repeatedly ($m-1$ times) Theorem 8.5 of \cite{dirac} (or Theorem \ref{fibres:of:lambda:thm} in the case 
$k=1$), we obtain that the fibre of $X$ at the wedge point has dimension $m$.
\end{example}

\subsection{The volume forms}

The notion of a \textbf{volume form} is well-defined for (a subcategory of) diffeological spaces (\cite{iglesiasBook}, Section 6.44). After recalling the necessary definitions, we consider its behavior under 
gluing. Let $X$ be a diffeological space of dimension $n$. A volume form on it is then a nowhere vanishing $n$-form on $X$; alternatively, it is a collection of usual volume forms on the domains of definition 
of plots of $X$.

\begin{defn}
Let $X$ be a diffeological space, and let $n=\dim(X)$. A \textbf{volume form} on $X$ is a form $\omega\in\Omega^n(X)$ such that for every $x\in X$ there exists a plot $p:U\to X$ of $X$ such that $p(U)\ni x$
and $\omega(p)$ is a volume form on $U$.
\end{defn}

An alternative way to define a volume form is to ask that, for any $x\in X$, there be a plot $p$ such that $p(0)=x$ and $\omega(p)(0)\neq x$ (see \cite{iglesiasBook}, p. 158). As in the case of smooth manifolds, 
volume forms do not always exist (obviously, any non-orientable smooth manifold considered with its standard diffeology is an instance of a diffeological space that does not admit any). A characterization of 
volume forms on $X_1\cup_f X_2$ follows from the definition and the characterization of the space $\Omega^n(X_1\cup_f X_2)$ given above (Corollary \ref{omega:of:glued:cor}).

\begin{lemma}
Let $X_1$ and $X_2$ be two diffeological spaces of the same finite dimension $n$, let $f:X_1\supseteq Y\to X_2$ be a smooth map, and suppose that both $X_1$ and $X_2$ admit volume forms and that such 
forms can be chosen to be compatible. Then $X_1\cup_f X_2$ admits a volume form.
\end{lemma}

\begin{proof}
By assumption and Lemma \ref{dimension:of:glued:lem} we have $\dim(X_1\cup_f X_2)=n$. Let $\omega_1$ and $\omega_2$ be compatible volume forms on $X_1$ and $X_2$ respectively. It is then trivial 
to check that $\mathcal{L}^n(\omega_1\oplus\omega_2)$ is a volume form on $X_1\cup_f X_2$. Indeed, by Lemma \ref{dimension:of:glued:lem} $\dim(X_1\cup_f X_2)=n$. Let $x\in X_1\cup_f X_2$ be an 
arbitrary point. Then it has a lift to either $X_1$ or $X_2$ (possibly to both). Suppose that it has a lift $x_1\in X_1$; since $\omega_1$ is a volume form on $X_1$, there exists a plot $p_1$ of $X_1$ such that 
$\mbox{Range}(p_1)\ni x_1$ and $\omega_1(p_1)$ is a volume form on the domain of definition of $p_1$. Then $p=\tilde{i}_1\circ p_1$ is plot of $X_1\cup_f X_2$ such that $\mbox{Range}(p)\ni x$ and 
$\omega(p)=\omega_1(p_1)$ is a volume form on the domain of $p$. The case when $x$ has a lift to $X_2$ is treated analogously, so we obtain the claim.
\end{proof}

\begin{cor}
An instance of a volume form on $X_1\cup_f X_2$ is the form $\mathcal{L}^n(\omega_1\oplus\omega_2)$, where $\omega_1$ and $\omega_2$ are compatible volume forms on $X_1$ and $X_2$ respectively. 
\end{cor}

It is not clear whether the \emph{vice versa} of this statement is always true; we can only obtain it under some rather restrictive assumptions.

\begin{prop}\label{splitting:volume:form:on:glued:prop}
Let $X_1$ and $X_2$ be two diffeological spaces of finite dimension and such that $\dim(X_1)=\dim(X_2)$, let $f:X_1\supseteq Y\to X_2$ be a smooth map such that $i_2(f(Y))$ is D-open in $X_1\cup_f X_2$, 
and let $\omega$ be a volume form on $X_1\cup_f X_2$ such that $\pi^*(\omega)=\omega_1\oplus\omega_2$. Then $\omega_1$ and $\omega_2$ are volume forms on $X_1$ and $X_2$ respectively.  
\end{prop}

\begin{proof}
Let $x_1\in X_1$ be any point, and let $x=\pi(x_1)\in X_1\cup_f X_2$. Let $p:U\to X_1\cup_f X_2$ be a plot such that $p(U)\ni x$ and $\omega(p)$ is a volume form on $U$. The assumption that $i_2(f(Y))$ 
allows us to claim that $p$ has a lift to a plot $p_1$ of $X_1$ (which does not have to be true in the case of $x_1\in Y$). Then $p_1(U)\ni x_1$ and $\omega_1(p_1)$ is a volume form on $U$. Since $x_1$ 
is an arbitrary point, we conclude that $\omega_1$ is a volume form on $X_1$. The case of $\omega_2$ is treated analogously.
\end{proof}

\begin{rem}
In the above proposition, let $n$ be the dimension of $X_1\cup_f X_2$. Recall that $\omega_1=\tilde{i}_1^*(\omega)\in\Omega^n(X_1)$ and $\omega_2=i_2^*(\omega)\in\Omega^n(X_2)$ respectively. The 
assumption that $\dim(X_1)=\dim(X_1\cup_f X_2)=\dim(X_2)$ was not really used in the proof of Proposition \ref{splitting:volume:form:on:glued:prop}; rather, we could obtain this equality as part of the 
conclusion. Notice also that Example \ref{dim:X:and:dim:lambda:are:different:ex} implies that there might be many volume forms on $X_1\cup_f X_2$ that are not proportional.
\end{rem}

\subsection{$\bigwedge^k(\Lambda^1(X))$ and $\Lambda^k(X)$ are not diffeomorphic}

Let $X$ be a finite-dimensional diffeological space, and let $k\geqslant 2$. We now show that $\Lambda^k(X)$ and $\bigwedge^k(\Lambda^1(X))$ are in general not the same. 

\begin{example}
Let $X$ be the wedge at the origin of two copies of the standard $\matR^2$, endowed with the corresponding gluing diffeology. Then by Theorem \ref{fibres:of:lambda:thm} we have that 
$\Lambda_0^1(X)\cong\matR^2\oplus\matR^2\cong\matR^4$. Since by construction $\bigwedge^2(\Lambda_0^1(X))=\left(\bigwedge^2(\Lambda^1(X))\right)_0$, the fibre of $\bigwedge^2(\Lambda^1(X))$ 
at the wedge point, this is a space of dimension $6$. However, $\Lambda_0^2(X)\cong\Lambda_0^2(\matR^2)\oplus\Lambda_0^2(\matR^2)$ has dimension $2$. 
\end{example}

We conclude from the above example that $\bigwedge^k(\Lambda^1(X))$ is \emph{a priori} a much larger space than $\Lambda^k(X)$. We shall see next whether there is any other natural relation between 
the two, for instance, whether an element of $\bigwedge^k(\Lambda^1(X))$ determines naturally an element of $\Lambda^k(X)$.

Recall first that there is a well-defined notion of the \textbf{exterior product} $\omega\wedge\mu\in\Omega^{k+l}(X)$ of any two differential forms $\omega\in\Omega^k(X)$ and $\mu\in\Omega^l(X)$ 
(\cite{iglesiasBook}, Section 6.35), which is defined by setting
$$(\omega\wedge\mu)(p):=\omega(p)\wedge\mu(p)$$ for all plots $p$ of $X$.

\begin{lemma}\label{exterior:derivative:descends:to:lambda:lem}
The exterior derivative $\wedge:\Omega^k(X)\times\Omega^l(X)\to\Omega^{k+l}(X)$ induces a well-defined and smooth pseudo-bundle map 
$$\Lambda^k(X)\times_X\Lambda^l(X)\to\Lambda^{k+l}(X).$$
\end{lemma}

\begin{proof}
Let $x\in X$; we need to show that if at least one of $\omega,\mu$ vanishes at $x$ then $\omega\wedge\mu$ vanishes at $x$. Let $p$ be a plot of $X$ such that $p(0)=x$. Obviously, 
$$(\omega\wedge\mu)(p)(0)=\omega(p)(0)\wedge\mu(p)(0),$$ so if one of $\omega(p)(0)$, $\mu(p)(0)$ is zero then $(\omega\wedge\mu)(p)(0)=0$. The smoothness is immediate from the definitions of 
the respective diffeologies, so we obtain the claim.
\end{proof}

We thus can obtain the following.

\begin{lemma}\label{exterior:derivative:maps:to:lambda:lem}
The exterior derivative yields a well-defined pseudo-bundle map
$$\left(\Lambda^k(X)\right)\bigwedge\left(\Lambda^l(X)\right)\to\Lambda^{k+l}(X).$$
\end{lemma}

\begin{proof}
This is immediate from Lemma \ref{exterior:derivative:descends:to:lambda:lem} and the construction of the diffeology on the exterior product of pseudo-bundle (based essentially on the properties of the 
tensor product diffeology, see \cite{wu}). 
\end{proof}

The obvious consequence of Lemma \ref{exterior:derivative:maps:to:lambda:lem} is the following statement. 

\begin{cor}
There is a well-defined pseudo-bundle map 
$$\bigwedge^{1,k}:\bigwedge^k(\Lambda^1(X))\to\Lambda^k(X)$$ induced by the exterior derivative.
\end{cor}

As follows from the example given in this section, the map $\bigwedge^{1,k}$ is in general not injective. It is not clear whether it is surjective.

\subsection{The Hodge star operator does not take values in $\bigwedge^{n-k}(\Lambda^1(X))$} 

For a diffeological vector space $V$ of finite dimension $n$, the standard definition of the Hodge star $\star:\bigwedge^kV\to\bigwedge^{n-k}V$ by setting 
$$e_{\sigma(1)}\wedge\ldots\wedge e_{\sigma(k)}\mapsto\mbox{sgn}(\sigma)e_{\sigma(k+1)}\wedge\ldots\wedge e_{\sigma(n)}$$ for all $\sigma\in\mbox{Symm}(n)$ and for all $k=1,\ldots,n$, where 
$\{e_i\}_{i=1}^n$ is a fixed basis of $V$ (we avoid the requirement of it being an orthonormal basis), yields a well-defined operator that is smooth for the natural diffeology on $\bigwedge^kV$ (induced by the 
tensor product diffeology). Thus, if $\pi:V\to X$ is a finite-dimensional diffeological vector pseudo-bundle that is locally trivial in the standard sense (in particular, it admits a local basis of smooth sections) then 
the operator $\star$ is defined on each $\bigwedge^k(V)$ for $k=1,\ldots,n=\dim(V)$, where $\dim(V)$ is the maximum of the usual vector space dimensions of fibres of $V$. 

Let now $X$ be a diffeological space of finite dimension such that $\Lambda^1(X)$ is finite-dimensional and admits pseudo-metrics; let $g^{\Lambda}$ be a fixed pseudo-metric on $\Lambda^1(X)$. Then for 
all $x\in X$ the fibre $\Lambda_x^1(X)$ admits an orthonormal, with respect to $g^{\Lambda}(x)$, basis $\alpha_{1,x},\ldots,\alpha_{n,x}$, with respect to which the map $\star_x$ is obviously defined. The 
collection of the maps $\star_x$ for all $x\in X$ yields in a usual way the operator $\star$ on $\bigwedge^k(\Lambda^1(X))$. However, it does not take values in $\bigwedge^{n-k}(\Lambda^1(X))$, as the 
next example shows.

\begin{example}\label{star:undefined:ex}
Let $X$ be the wedge at the origin of two copies, denoted by $X_1$ and $X_2$, of $\matR^2$, endowed with the gluing diffeology; then $\dim(X)=2$. The fibres of 
$\bigwedge^1(\Lambda^1(X))=\Lambda^1(X)$ can be described as follows:
$$\Lambda_x^1(X)=\left\{\begin{array}{ll} 
\mbox{Span}(dx_1,dy_1) & \mbox{if }\in\tilde{i}_1(X_1\setminus\{(0,0)\}) \\
\mbox{Span}(dx_1\oplus dx_2,dx_1\oplus dy_2,dy_1\oplus dx_2,dy_1\oplus dy_2) & \mbox{if }x=(0,0) \\
\mbox{Span}(dx_2,dy_2) &  \mbox{if }\in i_2(X_2\setminus\{(0,0)\})
\end{array}\right.$$ Notice that endowing each fibre with the scalar product for the basis indicated is orthonormal yields a well-defined pseudo-metric on $\Lambda^1(X)$. 

Now, applying the standard construction of the Hodge star to each fibre yields a map that does \emph{not} take values in $\bigwedge^2(\Lambda^1(X))$. Indeed, on the fibre at the wedge point $(0,0)$ we 
would, by formal definition, have 
$$\star_{(0,0)}(dx_1\oplus dx_2)=(dx_1\oplus dy_2)\wedge(dy_1\oplus dx_2)\wedge(dy_1\oplus dy_2)\in\bigwedge^3(\Lambda_{(0,0)}^1(X)).$$ 
\end{example}

The example just made indicates that, at a minimum, the Hodge star is not readily defined on exterior degrees of $\Lambda^1(X)$.

\section{The De Rham operator on $\bigwedge(\Lambda^1(X))$}

We have established so far that there is no readily available counterpart of the standard operator $d+d^*$ in the diffeological context. Therefore the De Rham operator on $\bigwedge(\Lambda^1(X))$ 
(endowed with a pseudo-metric) can only be defined as the composition of the standard Clifford action with the Levi-Civita connection, assuming that the latter exists. Another assumption that is needed 
is that $\bigwedge(\Lambda^1(X))$ have only a finite number of components (summands of form $\bigwedge^k(\Lambda^1(X))$), that is, that there is a uniform bound on the dimensions of fibres of 
$\Lambda^1(X)$; as we have seen in the previous section, this is not implied by $X$ having finite dimension.

\subsection{Bounding the dimension of $\Lambda^1(X)$}

Let $X$ be a diffeological space of finite dimension. The next example shows that the set of the dimensions of fibres of $\Lambda^1(X)$ may not have a supremum.

\begin{example}
Consider the following sequence $\{X_n\}_{n\in\matN}$ of diffeological spaces: $X_0$ is the standard $\matR$, and if $X_n$ is already defined then $X_{n+1}$ is obtained as the wedge of $X_n$ at the 
point $x=n+1\in X_n$ with $n+1$ copies of the standard $\matR$ at zero of each copy; formally, $X_{n+1}$ is the result of a sequence of $n+1$ gluings of $X_{n+1}^{(k)}$ (this is $X_{n+1}$ to which $k$ 
copies of $\matR$ have already been added) to the standard $\matR$ along the map $\{n\}\to\{0\}$. Each $X_n$ is thus endowed with a well-defined diffeology based on the gluing construction; furthermore, 
there is a sequence of smooth inclusions $X_0\subset X_1\subset\ldots\subset X_n\subset\ldots$. Let $X=\bigcup_{n\in\matN}X_n$; endow it with the minimal diffeology such that all these inclusions are smooth 
(the diffeology of $X$ is essentially the union of the diffeologies of all $X_n$ and can be called the inductive limit diffeology). 

Since all the gluing points are isolated and the differential forms are local, we obtain that $\Lambda_{(n)}^1(X)\cong\Lambda_{(n)}^1(X_n)$ and in particular has dimension $n$, by the reasoning made 
already. Thus, $\Lambda^1(X)$ admits fibres of arbitrarily large dimension, although the dimension of $X$ itself is equal to $1$.
\end{example}

The above example shows that the existence of $\max\{\dim(\Lambda_x^1(X))\}$ should be imposed as a separate assumption. If such a maximum exists, we say that $\Lambda^1(X)$ has \textbf{bounded 
dimension}.

\subsection{The definition of the De Rham operator}

Let $X$ be a diffeological space of finite dimension and such that the following two conditions hold. First, $\Lambda^1(X)$ has bounded dimension; second, $\Lambda^1(X)$ admits a pseudo-metric 
$g^{\Lambda}$ such that there exists a Levi-Civita connection $\nabla^{\Lambda}$ on $(\Lambda^1(X),g^{\Lambda})$. Let $n=\max_{x\in X}\{\dim(\Lambda_x^1(X))\}$, and consider 
$$\bigwedge(\Lambda^1(X)):=\bigoplus_{k=0}^n\bigwedge^k(\Lambda^1(X)),$$ which is a diffeological vector pseudo-bundle for its standard diffeology based on the tensor product diffeology. Then 
the standard Clifford action $c^{\Lambda}$ of $\cl(\Lambda^1(X),g^{\Lambda})$ on $\bigwedge(\Lambda^1(X))$ is smooth (\cite{pseudo-bundles-exterior-algebras}); furthermore, the connection 
$\nabla^{\Lambda}$ induces (in a completely standard way) a connection $\nabla^{\bigwedge(\Lambda)}$ on $\bigwedge(\Lambda^1(X))$. 

\begin{defn}
Let $X$ be a diffeological space satisfying the above conditions. The \textbf{diffeological De Rham operator} on $(X,g^{\Lambda})$ is the operator 
$$D_{dR}:C^{\infty}(X,\bigwedge(\Lambda^1(X)))\to C^{\infty}(X,\bigwedge(\Lambda^1(X)))$$ defined as the composition
$$D_{dR}=c^{\Lambda}\circ\nabla^{\bigwedge(\Lambda)}.$$
\end{defn}

\begin{example}
Let $X$ be a wedge of two copies of the standard $\matR^2$ at the origin; endow $X$ with the corresponding gluing diffeology and denote the two copies of $\matR^2$ by $X_1$ and $X_2$ respectively. 
Each of the two spaces $\Lambda^1(X_1)$, $\Lambda^1(X_2)$ can thus be standardly identified with $T^*\matR^2$, and every fibre written in the form, 
$\Lambda_{(x_1,y_1)}^1(X_1)=\mbox{Span}(dx_1,dy_1)$ and $\Lambda_{(x_2,y_2)}^1(X_2)=\mbox{Span}(dx_2,dy_2)$. Let $g_i^{\Lambda}$ be the pseudo-metric on $\Lambda^1(X_i)$ given by 
$$g_i^{\Lambda}(x_i,y_i)=\frac{\partial}{\partial x_i}\otimes\frac{\partial}{\partial x_i}+e^{x_iy_i}\frac{\partial}{\partial y_i}\otimes\frac{\partial}{\partial y_i}\,\,\mbox{ for }i=1,2,$$ where $\frac{\partial}{\partial x_i}$ 
is the dual map of $dx_i$ and $\frac{\partial}{\partial y_i}$ is the dual of $dy_i$. Let $g^{\Lambda}$ be the corresponding induced pseudo-metric on $\Lambda^1(X)$. Notice that $g_1^{\Lambda}$ and 
$g_2^{\Lambda}$ are automatically compatible, since all the compatibility conditions are empty in the case of gluing along a single-point set; in particular, we have 
$$g^{\Lambda}(0,0)=\frac12(\frac{\partial}{\partial x_1}\otimes\frac{\partial}{\partial x_1}+\frac{\partial}{\partial y_1}\otimes\frac{\partial}{\partial y_2}+
\frac{\partial}{\partial x_2}\otimes\frac{\partial}{\partial x_2}+\frac{\partial}{\partial y_2}\otimes\frac{\partial}{\partial y_2}).$$

Every fibre of $\bigwedge(\Lambda^1(X))$ outside of the wedge point coincides with $\bigwedge(\matR^2)$, while at the wedge point it is $\bigwedge(\matR^2\oplus\matR^2)$. The Clifford algebra 
$\cl(\Lambda^1(X),g^{\Lambda})$ behaves as $\cl(\Lambda^1(X_i),g_i^{\Lambda})$, for the appropriate $i=1,2$, outside of the wedge point. At the wedge point it is equivalent to 
$\cl(\matR^4,\langle\cdot,\cdot\rangle)$, where $\langle\cdot,\cdot\rangle$ is the canonical scalar product. The Clifford action $c^{\Lambda}$ is standardly defined; for instance, 
$$c^{\Lambda}(dx_1)(dy_2)=dx_1\wedge dy_2-\frac12.$$

The sections of $\Lambda^1(X)$ are in one-to-one correspondence with pairs of sections of $\Lambda^1(X_1)$ and $\Lambda^1(X_2)$: if $s\in C^{\infty}(X,\Lambda^1(X))$ then 
$s_1=\tilde{\rho}_1^{\Lambda}\circ s\circ\tilde{i}_1\in C^{\infty}(X_1,\Lambda^1(X_1))$ and $s_2=\tilde{\rho}_2^{\Lambda}\circ s\circ i_2\in C^{\infty}(X_2,\Lambda^1(X_2))$. \emph{Vice versa}, and this is 
specific to the present instance, given $s_1\in C^{\infty}(X_1,\Lambda^1(X_1))$ and $s_2\in C^{\infty}(X_2,\Lambda^1(X_2))$, we set $s(0,0)=s_1(0,0)\oplus s_2(0,0)$, while outside the wedge point 
$s$ is equivalent to either $s_1$ or $s_2$ in the obvious sense. 

Both $X_1$ and $X_2$ are endowed with the standard Levi-Civita connections $\nabla^{\Lambda,1}$ and $\nabla^{\Lambda,2}$ respectively. These induce the Levi-Civita $\nabla^{\Lambda}$ on $X$ 
(relative to the induced pseudo-metric $g^{\Lambda}$); for a section $s$ of $\Lambda^1(X)$ determined by a pair of sections $s_1\in C^{\infty}(X_1,\Lambda^1(X_1))$ and 
$s_2\in C^{\infty}(X_2,\Lambda^1(X_2))$, $\nabla^{\Lambda}s$ coincides (up to appropriate identifications) with either $\nabla^{\Lambda,1}s_1$ or $\nabla^{\Lambda,2}s_2$, while at the wedge point 
its value is essentially $(\nabla^{\Lambda,1}s_1)(0,0)\oplus(\nabla^{\Lambda,2}s_2)(0,0)$ (a formalization of this construction is available in \cite{connectionsLC}). A fully standard procedure completes 
the construction.
\end{example}

\begin{rem}
In a previous section we indicated that one (but not the only one) problem in defining the Hodge star for diffeological spaces is that the standard definition does not, in general, yield a map 
$\bigwedge^k(\Lambda^1(X))\to\bigwedge^{n-k}(\Lambda^1(X))$ for a fixed $n$ independent of $k$. It follows that there might be a way to define $\star$ as taking values in $\bigwedge(\Lambda^1(X))$ if 
$\Lambda^1(X)$ has bounded dimension. Since this was not the only difficulty in extending the definition of $d+d^*$ to the diffeological context (recall that already $d$ does not descend to a pseudo-bundle 
map on $\Lambda^1(X)$), we do not go in that direction for now.
\end{rem}

\subsection*{Appendix: on the possibility of the De Rham-like operator $d^*+d^{**}$}

We briefly consider here the possibility of defining a De Rham-like operator $d^*+d^{**}$, based on the map $d^*:(\Omega^{k+1}(X))^*\to(\Omega^k(X))^*$ dual to the differential. This construction comes from 
the following observations. One, each dual pseudo-bundle $(\Lambda^k(X))^*$ embeds into the trivial pseudo-bundle $X\times(\Omega^k(X))^*$ via the map $(\pi^{\Omega^k,\Lambda^k})^*$ that is the map 
dual to the defining projection $\pi^{\Omega^k,\Lambda^k}:X\times\Omega^k(X)\to\Lambda^k(X)$ (the dual map is an embedding simply because $\pi^{\Omega^k,\Lambda^k}$ is surjective, and by definition 
of the dual pseudo-bundle diffeology). Suppose for the moment that we have 
$$d^*\left((\pi^{\Omega^{k+1},\Lambda^{k+1}})^*((\Lambda^{k+1}(X))^*)\right)\subseteq(\pi^{\Omega^k,\Lambda^k})^*((\Lambda^k(X))^*);$$ then $d^*$ is well-defined as a map
$$d^*:(\Lambda^{k+1}(X))^*\to(\Lambda^k(X))^*.$$

The second observation is that if $\Lambda^k(X)$ has only finite-dimensional fibres, then all of these fibres are standard; so if each $\Lambda^k(X)$ is endowed with a pseudo-metric $g^{\Lambda^k}$ 
then the corresponding pairing map $\Phi_{g^{\Lambda^k}}$ is a diffeomorphism $\Lambda^k(X)\to(\Lambda^k(X))^*$. Therefore the composition
$$(\Phi_{g^{\Lambda^k}})^{-1}\circ d^*\circ\Phi_{g^{\Lambda^{k+1}}},$$ also denoted by $d^*$, is a well-defined smooth operator $\Lambda^{k+1}(X)\to\Lambda^k(X)$. The corresponding dual map 
$$(d^*)^*:(\Lambda^k(X))^*\to(\Lambda^{k+1}(X))^*$$ likewise provides us with a well-defined operator
$$d^{**}:=(\Phi_{g^{\Lambda^{k+1}}})^{-1}\circ(d^*)^*\circ\Phi_{g^{\Lambda^k}}:\Lambda^k(X)\to\Lambda^{k+1}(X).$$ It then suffices to assume that $X$ has finite dimension $n$ to obtain a well-defined 
De Rham-like operator on $\Lambda(X):=\bigoplus_{k=0}^n\Lambda^k(X)$ which is
$$d^*+d^{**}:\Lambda(X)\to\Lambda(X).$$ 

Let us now consider the potential inclusion 
$$d^*\left((\pi^{\Omega^{k+1},\Lambda^{k+1}})^*((\Lambda^{k+1}(X))^*)\right)\subseteq(\pi^{\Omega^k,\Lambda^k})^*((\Lambda^k(X))^*).$$ Let $\alpha_{k+1}^*\in(\Lambda^{k+1}(X))^*$; consider 
$$d^*((\pi^{\Omega^{k+1},\Lambda^{k+1}})^*(\alpha_{k+1}^*))(\omega_k)=\alpha_{k+1}^*(d\omega_k+\Omega_x^{k+1}(X)).$$ Under the assumption that all $\Lambda_x^1k(X)$ are finite-dimensional, it 
is the image of some $\alpha_k^*\in(\Lambda^k(X))^*$ if and only if 
$$d\omega_k+\Omega_x^{k+1}(X)\in\mbox{Ker}(\alpha_{k+1}^*)\,\,\mbox{ for all }\,\omega_k\in\Omega_x^k(X).$$ Although this is a less restrictive condition than $d\omega_k\in\Omega_x^{k+1}(X)$, there is 
no obvious reason for it to hold \emph{a priori}; it essentially requires that each element of $(\Lambda_x^{k+1}(X))^*$ vanish on the image $\pi^{\Omega^{k+1},\Lambda^{k+1}}(B_{dR}^{k+1}(X))$ of the space 
of the coboundaries. We thus conclude that the operator $d^*+d^{**}$ might be defined, not on the entire pseudo-bundle $\oplus\Lambda^k(X)$, but rather on its reduction by the complex of the coboundaries, 
by which we mean the pseudo-bundle obtained by taking, instead of $\Lambda^k(X)=(X\times\Omega^k(X))/(\cup_{x\in X}(\{x\}\times\Omega_x^k(X)))$, its quotient pseudo-bundle 
$$\Lambda_{dR}^k(X):=(X\times\Omega^k(X))/(\cup_{x\in X}(\{x\}\times\mbox{Span}(\Omega_x^k(X),B_{dR}^k(X)))).$$ We leave for other work the question of whether the construction thus obtained would 
be anything other than trivial.

\vspace{1cm}

\noindent University of Pisa \\
Department of Mathematics \\
Via F. Buonarroti 1C\\
56127 PISA -- Italy\\
\ \\
ekaterina.pervova@unipi.it\\

\end{document}